\documentclass[9.5pt,reqno]{amsart}
\usepackage{amssymb,mathrsfs,graphicx}
\usepackage{epsfig,subfigure}
\usepackage{ifthen}
\usepackage{refcheck}
\usepackage{float}
\usepackage{colortbl}

\usepackage{ifthen} 
\provideboolean{shownotes} 
\setboolean{shownotes}{true} 
\newcommand{\margnote}[1]{
\ifthenelse{\boolean{shownotes}}%
{\marginpar{\raggedright\tiny\texttt{#1}}}%
{}%
}

\newcommand{\comment}[1]{{\color{red} #1}}

\newcommand{\hole}[1]{
\ifthenelse{\boolean{shownotes}}%
{\begin{center} \fbox{ \rule {.25cm}{0cm} \rule[-.1cm]{0cm}{.4cm}
\parbox{.85\textwidth}{\begin{center} \texttt{#1}\end{center}} \rule
{.25cm}{0cm}}\end{center}} {} }

\title[Propagation of chaos for some 2D fractional Keller Segel equations]{Propagation of chaos for some 2 dimensional fractional Keller Segel equations in diffusion dominated and fair competition cases}

\author[Salem]{Samir Salem}
\address[Samir Salem]{\newline CEREMADE UMR 7534, \newline
    Universit\'e Paris Dauphine, Place du Mar\'echal de Tassigny, Cedex Paris, France}
\email{salem@ceremade.dauphine.fr}

\numberwithin{equation}{section}

\newtheorem{theorem}{Theorem}[section]
\newtheorem{lemma}{Lemma}[section]
\newtheorem{corollary}{Corollary}[section]
\newtheorem{proposition}{Proposition}[section]
\newtheorem{remark}{Remark}[section]
\newtheorem{definition}{Definition}[section]

\newcommand{\R}{\mathbb R}

\newcommand{\B}{\mathbb{B}}
\newcommand{\HH}{\mathcal{H}}
\newcommand{\II}{\mathcal{I}}

\newcommand{\MM}{\mathcal{M}}

\newcommand{\XX}{\mathcal{X}}

\newcommand{\PP}{\mathcal{P}}
\newcommand{\FF}{\mathcal{F}}
\newcommand{\CC}{\mathcal{C}}
\newcommand{\DD}{\mathcal{D}}
\newcommand{\ZZ}{\mathcal{Z}}

\newcommand{\mei}{\frac{1}{N}\sum_{i=1}^N}

\newcommand{\e}{\varepsilon}

\newcommand{\lal}{\langle}
\newcommand{\ral}{\rangle}

\newcommand{\lt}{\left}
\newcommand{\rt}{\right}
\newcommand{\pa}{\partial}
 
\newcommand{\mb}{\mathbf{1}}

\newcommand{\bq}{\begin{equation}}
\newcommand{\eq}{\end{equation}}

\newcommand{\LL}{\mathcal{L}}

\newcommand{\E}{\mathbb{E}}
\def\charf {\mbox{{\text 1}\kern-.30em {\text l}}}

\begin{document}
\allowdisplaybreaks

\date{\today}



\begin{abstract} 
		This work deals with the propagation of chaos without cut-off for some two dimensional fractional Keller Segel equations. The diffusion considered here is given by the fractional Laplacian operator $-(-\Delta)^{\frac{a}{2}}$ with $a \in (1,2)$ and the singularity of the interaction is of order $|x|^{1-\alpha}$ with $\alpha\in ]1,a]$. In the case $\alpha\in (1,a)$ we give a complete propagation of chaos result, proving the $\Gamma$-l.s.c property of the fractional Fisher information, already known for the classical Fisher information, using a result of \cite{HM}. In the fair competition case (see \cite{CarCal}) $a=\alpha$, we only prove a convergence/consistency result in a sub-critical mass regime, similarly as the result obtained for the classical Keller-Segel equation in \cite{FJ}.

\end{abstract}

\maketitle \centerline{\date}

\tableofcontents

%
%
%
%
\section{Introduction}\label{intro}

The parabolic-elliptic Keller Segel equation has received a large attention from the kinetic equation community lately. This model deals with the chemotaxis of cells or bacteria evolving in a  environment, which they are able to modify in order to communicate with each other. More precisely the evolution of the density of bacteria $\rho_t$ and the concentration of chemotaxis $c_t$ is given by the equation 
\bq
\label{eq:KS }
	\begin{aligned}
		&\pa_t \rho_t+\chi\nabla \cdot \lt( \rho_t\nabla c_t \rt)=\Delta\rho_t.,\cr
		&-\Delta c_t= \rho_t ,
	\end{aligned}
\eq
where $\chi>0$ is a sensitivity parameter encoding the intensity of the aggregation. We refer to \cite{DBP} for a proper biological and mathematical motivation. This model has been extensively studied, especially in dimension $2$ which is the best understood and which makes particular biological sense in the context of bacteria motion. Some blow up phenomena are known to arise if the initial mass is too large \cite[Corollary 2.2]{DBP}, and global well posedness holds when the mass is small enough \cite{MisEn}. However the question of propagation of chaos for this model remains open. \newline
Some bacteria are known for their "run and tumble" motion, therefore their trajectories are better described by L\'evy flights than Brownian motion (see for instance \cite{Cal}). This inclines to replace the classical diffusion in the evolution equation of the density of bacteria with a fractional diffusion. We define the fractional Laplacian on $\R^2$ of exponent $a/2\in (0,1)$ for smooth function $u\in \mathcal{C}_c^\infty(\R^2)$
\[
\begin{split}
-(-\Delta)^{a/2}u(x)&=c_{2,a}\mbox{v.p.}\int_{\R^2}\frac{u(x)-u(y) }{|x-y|^{2+a}}dy\\
&\quad =c_{2,a}\mbox{v.p.}\int_{\R^2}\frac{u(x)-u(y)-(x-y)\cdot \nabla u(x) }{|x-y|^{2+a}}dy,
\end{split}
\]
where $c_{2,a}$ is a normalization constant defined as
	\[
	c_{2,a}=-\frac{2^a \Gamma\lt(1+ \frac{a}{2} \rt)}{\pi \Gamma\lt( -\frac{a}{2} \rt)}=\frac{2^{a}a^2\Gamma^2\lt( \frac{a}{2} \rt)}{4\pi^2}\sin \lt(a \frac{\pi}{2}\rt),
	\]
	(see \cite{K} for equivalent definitions of the fractional Laplacian). Not only for the purpose of modeling, but also because of the recent popularity of fractional diffusion equation, the problem  
\bq
\label{eq:frNew1}
\begin{aligned}
	&\pa_t \rho_t+\chi\nabla \cdot \lt( \rho_t\nabla c_t \rt)+(-\Delta)^{\frac{a}{2}}\rho_t=0,\cr
	&-\Delta c_t= \rho_t ,
\end{aligned}
\eq
has been studied under various perspectives by different authors. In \cite{HL}, Huang and Liu obtained local in time existence for $L^2$ initial condition when $a\in (1,2)$ in dimension $2$. Escudero obtained global existence for a similar system in dimension $1$ in \cite{Esc}. In dimension $1$, equation \eqref{eq:frNew1} has also been studied by Clavez and Bournaveas in \cite{Cal} who prove global existence in case $a\in (0,1]$ for some initial condition in $L^p(\R)$ for some $p\geq \frac{1}{a}$ and in case $a\in (1,2]$. They also show blow-up in case $a\in (0,1)$ if the initial condition has small first order moment compared to initial mass. More recently, in $2$ dimensional settings and for $a\in(0,2)$, Biler et al obtained a blow-up condition for the solution \eqref{eq:frNew1} in \cite[Theorem 2.1]{Biot} for large $M^{\frac{2}{a}}$-Morrey norm of the initial condition.  \newline
In this paper we address the question of propagation of chaos, (we refer to \cite{Szn} and next section for details) for a similar equation as \eqref{eq:frNew1}, where we replace the Newtonian attraction force with a less singular interaction kernel. Let $\alpha\in (0,2)$, and on $\R^2$ define 
$$W_{\alpha}(x)=\frac{|x|^{2-\alpha}}{2-\alpha}, \quad \mbox{and}\quad K_{\alpha}(x)=-\nabla W_a(x)=-\frac{x}{|x|^{\alpha}},$$
(with the convention "$\frac{|x|^0}{0}=\ln |x|$"). For $(a,\alpha)\in (0,2)\times(0,2)$ and $N\geq 1$ let $(\ZZ_t^i)_{i=1,\cdots,N,t\geq 0}$ be $N$ independent $a$-stable L\'evy flights on $\R^2$ (more precisions will be given about $a$-stable process in the next session), $(X^1_0,\cdots,X^N_0)$ a random variable on $\R^{2N}$ independent of the $N$ L\'evy flights and consider the particle system evolving on the plane defined as
\bq
\label{eq:part_fr_KS}
X_t^i=X_0^i+\frac{\chi}{N}\int_0^t\sum_{j\neq i}K_{\alpha}(X_s^i-X_s^j)ds+\ZZ_t^i.
\eq
We expect that when the number of particle goes to infinity, and the family of initial condition is assumed to be $\rho_0$-chaotic (see \cite[Definition 2.1]{Szn}), the above particle system well approximates the following nonlinear PDE
\bq
\label{eq:Fr_KS}
\begin{split}
&\pa_t \rho_t+\chi\nabla \cdot \lt( \rho_t(K_{\alpha}*\rho_t)\rt)+\lt(-\Delta\rt)^{\frac{a}{2}}\rho_t=0\\
&\rho_{t=0}=\rho_0.
\end{split}
\eq
This nonlinear conservation equation is the equation satisfied by the time marginals of the process solution to the nonlinear SDE
\bq
\label{eq:NLSDE}
\XX_t=\XX_0+\chi\int_0^t\int_{ \R^{2} }K_{\alpha}(\XX_s-y)\rho_s(dy)\, ds+
\ZZ_t,\, \quad \rho_s=\LL(\XX_s), 
\eq
with $(\ZZ_t)_{t\geq 0}$ an $a$-stable L\'evy flight on $\R^2$ independent of $\XX_0$. \newline
In the rest of the paper, $\mbox{v.p.}$ stands for the Cauchy principal value of a singular integral, $\|\cdot\|_{L^p}$ is the $L^p$ norm on $\R^2$, $\|\cdot\|_{L^p \cap L^q}:=\|\cdot\|_{L^p}+\|\cdot\|_{L^q}$ and $|\cdot|_{H^{s}(\R^2) }$ is the fractional Sobolev semi-norm of exponent $s\in (0,1)$ defined as
\[
|u|_{H^{s}(\R^2) }:=c_{2,a}\int \int \frac{|u(x)-u(y)|^2}{|x-y|^{2+a}}dxdy.
\]
The Euclidian inner product on $\R^2$ will be denoted either $x\cdot y$ or $\left\langle x,y \right \rangle$, and sometimes the latest notation will be used for the inner product on other spaces when it will make sense. We will also use the notation $\lal x \ral=\sqrt{1+|x|^2}$, and for $\kappa\in \R$, $m_\kappa(x)=\lal x \ral^\kappa$  . The unit ball on $\R^2$ will be denoted $\B$, its complementary in the plan will be denoted $\B^c$, and the ball of radius $r>0$, $\B_r$. \\
For a functional $F$ on $\R^{2N}$, and $i=1,\cdots,N$ we note $\nabla_i F=(\pa_{2i-1} F,\pa_{2i}F) \in \R^2$, and the notation $X^x_k$ will stand for the integration variable $(x_1,\cdots,x_{k-1},x,x_{k+1},\cdots,x_N)$ in $\R^{2N}$ .\\
The notation $\PP(E)$ stands for the set of probability measures on $E$ and $\PP_{\mbox{sym}}(E^k)$ for the set of symmetric probability measures on $E^k$ and $\PP_{\mbox{sym}}(E^N)$ for the set of sequences of symmetric probabilities on $E^N$. $W_p$ is the Wasserstein metric of order $p\geq 1$. The notation, $\DD(0,T;\R^2)$ stands for the \textit{c\`adl\`ag} paths on $\R^2$. We define $\mathbf{e}_t:\gamma\in \DD(0,T;\R^2)\mapsto \gamma(t)\in \R^2$ the evolution map at time $t$, and for $\rho\in\PP\lt(\DD(0,T;\R^2)\rt)$ we implicitly associate the family of probability measures $(\rho_t\in \PP(\R^2))_{t\in[0,T]}$ defined as $\rho_t=\mathbf{e}_t\# \rho$. \\
Finally, the dependence of some generic constant $C$ on the parameters of the problem will always be expressed in its indexation. We set the notation $\Phi$ for the functional defined on the quarter plan $(0,\infty)\times (0,\infty)$ as $\Phi(x,y)=(x-y)\ln\lt(\frac{x}{y}\rt)$.
%
%

\section{Preliminaries and main results}

\subsection{Propagation of chaos} In this paper, we address the question of the propagation of chaos for the particle system \eqref{eq:part_fr_KS}. For the sake of completeness we recall some basic notions on this topic, and refer to \cite{Szn} for some further explanations. We begin with the
\begin{definition}[Definition 2.1 in \cite{Szn}]
	\label{def:chaos}
	Let $(u_N)_{N\geq 1}$ be a sequence of symmetric probabilities on $E^N$ ($u_N\in \PP_{\mbox{sym} }(E^N) $), with $E$ some polish space. We say that $u_N$ is $u$-chaotic, with $u\in \PP(E)$, if for any $k\geq 2$ and $\phi_1,\cdots,\phi_k\in C_b(E)$ it holds
	\[
	\lim_{N\rightarrow +\infty}\int_{E^N}\bigotimes_{j=1}^k\phi_j u_N=\prod_{j=1}^k\lt(\int_{E}\phi_j u\rt).
	\]
\end{definition}
Note that such a sequence $(u_N)_{N\geq 1}$ converges toward $\delta_u \in \PP(\PP(E))$ in the following sense. For any $k\geq 1$, let be $u_N^k$ the marginal of $u_N$ on $E^k$ and
\[
u^{\otimes k}=\int_{\PP(E)}\rho^{\otimes k}
\delta_u(d\rho),
\]
the projection on $\PP(E^k)$ of $\delta_u$. Then
\[
u_N^k\underset{N\rightarrow+\infty}{\overset{*}{\rightharpoonup}}u^{\otimes k} \quad \mbox{in} \quad \PP(E^k).
\] 
Then we need the important
\begin{proposition}[Proposition 2.2 of \cite{Szn}]
	\label{prop:chaos} Let be $(u_N)_{N\geq 1}$ be a sequence of symmetric probabilities on $E^N$ ($E$ a polish space), $(X^1,\cdots,X^N)_{N\geq 1}$ a sequence of random vector of law $u_N$, and $\mu_N=\mei\delta_{X^i}$ the emprical measure associated to this vector. Then
	\begin{itemize}
		\item[$(i)$] $u_N$ is $u$ chaotic if and only if $(\mu_N)_{N\geq 1}$ converges in law (weakly in $\PP(E)$) toward $u\in \PP(E)$.
		\item[$(ii)$] The sequence of random variables $(\mu_N)_{N\geq 1}$ is tight if and only if the sequence of law of $X^1$ under $u_N$ is tight. 
	\end{itemize}
\end{proposition}
Our aim is therefore to prove that the dynamic \eqref{eq:part_fr_KS} propagates chaos i.e. that if one starts this dynamic from some initial condition which law is $\rho_0$-chaotic, the law of the solution at time $t>0$ to \eqref{eq:part_fr_KS} is $\rho_t$-chaotic, with $\rho_t$ the solution at time $t>0$ to \eqref{eq:Fr_KS}. Or equivalently, due to the above Proposition, to prove that
\[
\mu^N_0=\mei\delta_{X_0^i}\underset{N\rightarrow+\infty}{\overset{*,(\LL)}{\rightharpoonup}}\rho_0
\, \Rightarrow \,
\mu^N_t=\mei\delta_{X_i^t}\underset{N\rightarrow+\infty}{\overset{*,(\LL)}{\rightharpoonup}}\rho_t.
\]
Implicitly, such a statement requires a good knowledge of the well posedness of the limit problem \eqref{eq:Fr_KS}. When this knowledge is not available (typically when the singularity of the kernel $K_\alpha$ is too strong), one can not expect such a strong result as it does not make sens to talk about "the" solution at time $t>0$, $\rho_t$. However one can look at a weaker result of the type, if
\[
\mu^N_0=\mei\delta_{X_0^i}\underset{N\rightarrow+\infty}{\overset{*,(\LL)}{\rightharpoonup}}\rho_0,
\] 
then there exists a subsequence of the $\lt((\mu_t^N)_{t\in[0,T]}\rt)_{N\geq 1}$ converging in law toward some (possibly random) $\lt(\mu_t\rt)_{t\in[0,T]}\in\PP\lt(\DD(0,T;\R^2)\rt)$, which solves to \eqref{eq:Fr_KS} and such that $\mu_0=\rho_0$. We talk in this case, we talk of convergence/consistency rather than propagation of chaos. 
\subsection{$a$-stable processes with index $a\in(1,2)$}

Let $M$ be a Poisson random measure on $\R^+\times \R^2$ of intensity $Leb\times \lambda$ where $\lambda$ is a $\sigma$-finite measure on $\R^2$ satisfying $\int_{ \R^2 }(|x|^2\wedge 1)\lambda (dx)<\infty $ (see for instance \cite[Definition 2.3, Chapter V]{Cin}). Denote $\bar{M}$ its compensated measure, i.e. $\bar{M}(ds,dx)=M(ds,dx)-ds\lambda(dx)$, and denote $(Z_t)_{t\geq 0}$ the following L\'evy process
\bq
\label{eq:aLev}
Z_t=\int_{[0,t]\times\B}x \overline{M}(ds,dx)+\int_{[0,t]\times\B^c}x M(ds,dx).
\eq
The first stochastic integral in the r.h.s. converges in the sense of the Cauchy principal value i.e.
\[
\int_{[0,t]\times\B}x \overline{M}(ds,dx) =\lim_{\e \rightarrow 0}\lt( \int_{[0,t]\times\B\setminus \B_\e}x M(ds,dx)-t \int_{  \B\setminus \B_\e} x \lambda dx\rt) \quad \mbox{a.s.}
\]
Due to Ito's rule \cite[Theorem 4.4.7, p 226]{Ebe} we have for a test function $\phi$ smooth enough
$$\begin{aligned}
\phi(Z_t)=&\phi(Z_s)+\int_{[s,t]\times\B}x\cdot\nabla \phi(Z_{u_-})\overline{M}(du,dx)+\int_{[s,t]\times\B^c}x\cdot\nabla \phi(Z_{u_-})M(du,dx)\\
&\quad +\int_{[s,t]\times \R^2}\lt(\phi(Z_{u_-}+x)-\phi(Z_{u_-})-x\cdot\nabla \phi(Z_{u_-}) \rt)M(du,dx).
\end{aligned}$$
In the particular case 
$$\lambda(dx)=c_{2,a}\frac{dx}{|x|^{2+a}},$$
we can rewrite
\[
Z_t=\int_{[0,t]\times \R^2} x \bar{M}(ds,dx),
\]
since $\int_{\B^c} \lambda(dx)=0$, and then
\bq
\label{eq:martlev}
\begin{aligned}
	\phi(Z_t)=&\phi(Z_s) +\int_{[s,t]\times \R^2}\lt(\phi(Z_{u_-}+x)-\phi(Z_{u_-})) \rt)\bar{M}(du,dx)\\
	&\quad  +c_{2,a}\int_s^t\mbox{v.p.}\int_{\R^2}\frac{\phi(Z_{u}+x)-\phi(Z_{u})-x\cdot\nabla \phi(Z_{u})}{|x|^{2+a}}dx\,du.
\end{aligned}
\eq
This particular choice of intensity makes of $(Z_t)_{t\geq 0}$ defined in \eqref{eq:aLev} an $a$-stable L\'evy process, i.e. $(Z_t)_{t\geq 0}$ has the same law as $(u^{-1/a}Z_{ut})_{t\geq 0}$ for any $u>0$. Necessarily, such a process can only exists for $a\in[0,2]$ \cite[Exercice 2.34, Chapter VI]{Cin}, the case $a=0$ corresponding to the null process, and the case $a=2$, to the standard Brownian motion. It is well known, but we also see from \eqref{eq:martlev}, that the infinitesimal generator of the $a$-stable L\'evy process is the fractional Laplacian of exponent $a/2$. It deduces also from \eqref{eq:martlev} and classical properties of Poisson random measures that for any smooth function $\phi$ the process $(\mathcal{M}_t)_{t\geq 0}$ defined as
\[
\begin{aligned}
\mathcal{M}_t=\phi(Z_t)-\phi(0) -c_{2,a}\int_0^t\mbox{v.p.}\int_{\R^2}\frac{\phi(Z_{u}+x)-\phi(Z_{u})-z\cdot\nabla \phi(Z_{u})}{|z|^{2+a}}dx\,du,
\end{aligned}
\]
is a martingale, and the $a$-stable L\'evy flight $(Z_t)_{t\geq 0}$ is the only process such that $(\mathcal{M}_t)_{t\geq 0}$ defined as above is a martingale.

\subsection{Main result}
We now give some comments on propagation of chaos results for similar systems as the one studied here, already existing in the literature after which we will give the main result of this paper. We emphasize that there are only a few results of propagation of chaos for particle system with singular interaction and additive diffusion, beside the ones we recall here. They rely essentially on the fact that the diffusion is non degenerate, and in particular the strategy would not apply for second order system with a diffusion only in velocity.  \newline
We introduce three different cases. When $a>\alpha$ we say we are in \textit{Diffusion Dominated} case, when $a=\alpha$, in \textit{Fair competition} case and $a<\alpha$ in \textit{Aggregation Dominated} case. This terminology has been introduced by Carrillo et al (see for instance \cite[Definition 2.6]{CarCal}) and is based on the homogeneity analysis of the free energy for which the system they study is a gradient flow (in Wasserstein metric). However, the presence of fractional diffusion here makes difficult to write equation \eqref{eq:Fr_KS} in a gradient flow form (in usual Wasserstein metric see for instance \cite{Erb}), and in some sense we abuse their terminology. Nevertheless note that the classical $2$ dimensional Keller-Segel (which is $a=\alpha=2$) falls in Fair competition case, as they define it, as the sub-critical Keller Segel equation (i.e. $2=a>\alpha>1$) studied in \cite{GQ} falls in the Diffusion Dominated case. So that the extension we give here of their nomenclature is not completely senseless.\newline

We begin with the \textit{Aggregation Dominated} case, as it is the less understood of all. To the best of the author's knowledge, there is no result without cut-off, except for the case $a=0$ that is without any diffusion at all. In \cite{CarChoiHau}, the authors consider the case $a=0$ and $\alpha\in (0,2)$, as the absence of diffusion makes possible here to control the minimal inter-particle distance, thus the singularity of the interaction, but under some very restrictive assumptions on the initial distribution of the particles. As for cut-off result, Huang and Liu treated the case $\alpha=2$ and $a\in (0,1)$ with logarithmic cut-off of order $(\ln N)^{-1/2}$ in \cite{HL}. More recently Garcia and Pickl treated the classical Keller Segel case $a=\alpha=2$ with polynomial cut-off of order $N^{-\alpha}$ with $\alpha\in (0,1/2)$, but the coupled techniques they used could easily be extended to the full Aggregation Dominated case.\newline
For the \textit{Fair competition} case, the only existing result to the best of the author's knowledge, is the one of Fournier and Jourdain \cite[Theorem 6]{FJ} for the classical Keller Segel equation. Based on a strategy already used by Osada \cite{Osa}, the authors establish the convergence of the empirical measure associated to the particle system \eqref{eq:part_fr_KS} in case $a=\alpha=2$ to a solution to the corresponding limiting equation \eqref{eq:KS } The strategy relies on the simple following Ito's computation
\bq
\label{eq:FouJou}
\begin{split}
d\lt| X^1_t-X_t^2  \rt|^\e&= -\chi\e\lt| X^1_t-X_t^2  \rt|^{\e-2}\left\langle  X^1_t-X_t^2, \frac{1}{N} \sum_{k>2} \lt( K_2(X^1_t-X_t^k) -  K_2(X_t^2-X_t^k) \rt)\right\rangle dt \\
&\quad - \frac{2\chi \e}{N} \lt| X^1_t-X_t^2  \rt|^{\e-2}dt+ +2\e^2  \lt| X^1_t-X_t^2  \rt|^{\e-2}dt\\
&\quad+\e\lt| X^1_t-X_t^2  \rt|^{\e-2}\left\langle  X^1_t-X_t^2,dB_t^1-dB_t^2\right\rangle. 
\end{split}
\eq
As we are in the Fair competition case, the good term produced by the diffusion is of the same order as the bad terms produced by the drift, and choosing rightfully the exponent $\e\in (0,1)$ and the sensitivity $\chi>0$ enables to control 
$
\int_0^t\E\lt[|X_s^1-X_s^2|^{\e-2}\rt]\,ds,
$ 
and therefore the singularity of the interaction 
$
\int_0^t\E\lt[K_2(X_s^1-X_s^2)\rt]\,ds=\int_0^t\E\lt[|X_s^1-X_s^2|^{-1}\rt]\,ds.
$
However they are only able to conclude to a convergence/consistency result, as the limiting solution lies the class of weak solution to \eqref{eq:KS } which satisfies 
\[
\int_0^t \int_{\R^2\times \R^2}|x-y|^{-1}\rho_s(dx)\rho_s(dy) <\infty,
\]
for which uniqueness is not known. 
\newline
Finally the \textit{Diffusion Dominated} case is the easier as far as propagation of chaos is concerned. Godinho and Quininao treated the case $a=2,\alpha\in (1,2)$ in \cite{GQ}. They followed the strategy introduced by Fournier et al in \cite{MSN} for the propagation of chaos for the $2$D Navier-Stokes equation in vortex formulation. This strategy relies on a control of the entropy dissipation along the Liouville equation associated to the particle system \eqref{eq:part_fr_KS} (that is the equation solved by the law of this particle system). In the case $a=2$, this dissipation is the so called Fisher information, which enables to control the singularity of $K_{\alpha}$ for $\alpha \in (1,2)$ (see \cite[Lemma 3.3]{MSN}). Not only this information enables to deduce the convergence of the particle system, but passing to the limit, to deduce that the limiting solution lies in the class of solution which Fisher information is locally integrable in time, hence locally $L^1$ in time and $L^p$ in space for some $p\in \frac{2}{1-\alpha}$ (see \cite[Theorem 1.5]{GQ}). Uniqueness for equation \eqref{eq:Fr_KS} in this class is known in the case $\alpha<a=2$ (see \cite[Theorem 5.2]{GQ}), and so the authors manage to conclude to a full propagation of chaos result. Let also be mentioned that when $a=2$ and $\alpha \in(0,1)$, the interaction $K_{\alpha}$ is $(1-\alpha)$-Holder, and the propagation of chaos falls within a recent and more general result by Holding \cite[Theorem 2.1]{Hol}. This result provides convergence rate in Wasserstein metric, and is very different from the other results quoted in this section. Nevertheless, this result is obtained by taking advantage of the diffusion, and could not be stated in the deterministic case. \newline
To summarize, the Newtonian interaction is critically controlled by a classical diffusion, and a less singular than Newtonian interaction is perfectly controlled without any diffusion at all or with a classical diffusion. As these two cases correspond to the two extremal exponent for stable L\'evy process ($a=0$ and $a=2$), a natural question to investigate is, which type of singularity can be controlled by a fractional diffusion. It is the object of the main result of this paper given in the 

\begin{theorem}
	\label{thm:main}
	Let be $T>0$. 
	\begin{itemize}
		\item[Diffusion Dominated] Let be $2>a>\alpha>1$ and $\lt((X_t^1,\cdots,X_t^N)_{t\in[0,T]}\rt)_{N\geq 1}$ be a sequence of solutions to equation \eqref{eq:part_fr_KS} with initial condition of law $\lt(\rho_0^{\otimes N}\rt)_{N\geq 1}$ with $\rho_0\in L\log L(\R^2)\cap \PP_\kappa(\R^2)$ for some $\kappa\in(1,a)$. Then the sequence $\lt(\lt(\mei \delta_{X_t^i}\rt)_{t\in[0,T]}\rt)_{N\geq 1}$ goes in law $\PP\lt(\mathcal{D}(0,T;\R^2)\rt)$ to the unique solution $(\rho_t)_{t\in[0,T]}$  to equation \eqref{eq:Fr_KS} in $L^1 (0,T;L^{p}(\R^2))$ for any $p\in \lt(1,\frac{1}{1-a/2}\rt)$ starting from $\rho_0$. 
			\item[Fair Competition] Let be $2>a=\alpha>1$ and $\lt((X_t^1,\cdots,X_t^N)_{t\in[0,T]}\rt)_{N\geq 1}$ be a sequence of solutions to equation \eqref{eq:part_fr_KS} with initial condition of laws $\lt(F_0^N\rt)_{N\geq 1}\in \PP_{\mbox{sym}}(\R^{2N})$ being $\rho_0$-chaotic in the sense of Definition \ref{def:chaos} and satisfying \[
			\sup_{N\geq 1}\int_{\R^{2N}}\lal x_1\ral^{\kappa}F_0^N(dx_1,\cdots,dx_N)<\infty,
			\]  
			 for some $\kappa\in(1,a)$.
		 There exists $a^*\in(1,2)$ and $\chi_a:(a^*,2)\mapsto \R^*_+$ with $\lim_{a\rightarrow 2^-} \chi_a=1$ such that if $a\in (a^*,2)$ and $\chi\in (0,\chi_a)$, then there exists a subsequence of $\lt(\lt(\mei \delta_{X_t^i}\rt)_{t\in[0,T]}\rt)_{N\geq 1}$ which goes in law in $\PP(\DD(0,T;\R^2))$ to $(\rho_t)_{t \in [0,T]}$, a solution to equation \eqref{eq:Fr_KS} starting from $\rho_0\in \PP_{\kappa}(\R^2)$ which satisfies for any $\e\in (0,1)$
		\[
		\int_0^T\int_{\R^2\times \R^2} |x-y|^{\e-a}\rho_s(dx)\rho_s(dy)\,ds<\infty.
		\]
	
	\end{itemize}

\end{theorem}

We emphasize that this result is, to the best of the author's knowledge, the first propagation of chaos result with singular kernel and anomalous diffusion.\newline
Let us briefly sketch the proof of this theorem. For the Diffusion Dominated case, we follow the strategy of \cite{MSN}. In Proposition \ref{prop:tiDD} we show the tightness of the law of the particle system \eqref{eq:part_fr_KS}, which implies due to Proposition \ref{prop:chaos} the tightness of the law of the sequence of empirical measure $\lt((\mei\delta_{X_t^i})_{t\in [0,T]}\rt)_{N\geq 1}$. Therefore we can find subsequence converging in law toward a limit which turns out to be a solution to equation \eqref{eq:Fr_KS} which has a locally integrable in time fractional Fisher information. In Proposition \ref{prop:stab} we show that such a solution is unique, which concludes the desired propagation of chaos result.\newline
For the Fair Competition case, we follow the strategy of \cite{FJ}. In Proposition \ref{prop:tiFC} we also show  the tightness of the law of the sequence of empirical measure $\lt((\mei\delta_{X_t^i})_{t\in [0,T]}\rt)_{N\geq 1}$. This implies that there exists a subsequence converging in law toward some solution to equation \eqref{eq:Fr_KS}, in a class for which uniqueness is not known. Therefore we only conclude to a convergence/consistency result.

\begin{figure}[H]
	\centering
	\includegraphics[scale=0.2]{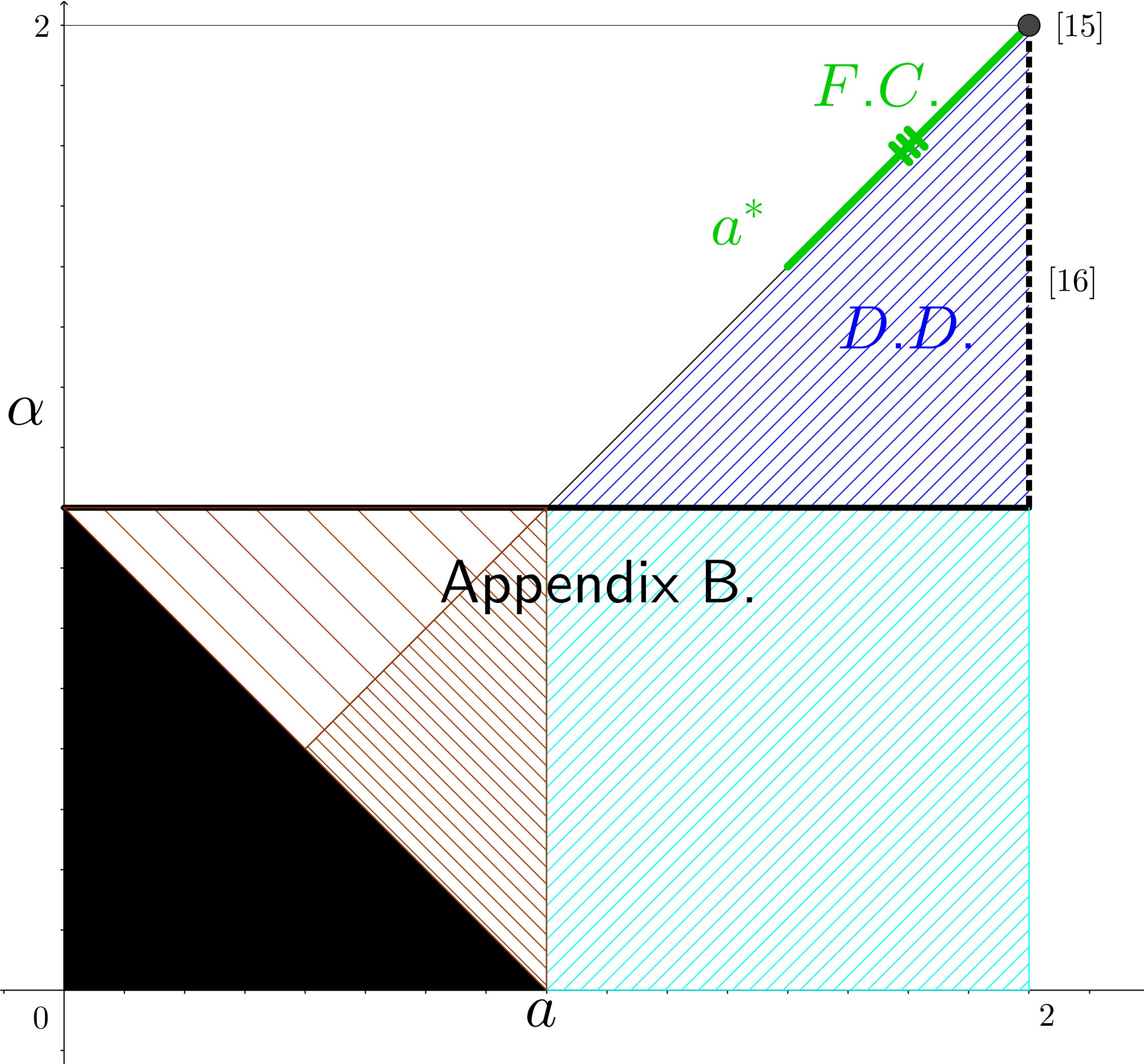}
	\caption{Existing results of propagation of chaos, without cut-off for equation \eqref{eq:Fr_KS}, together with the main result of this paper}
	\label{fig1}
\end{figure}

%
%
%
%
\section{Tightness estimate}

The key point of the proof of Theorem \ref{thm:main}, is the tightness of the law of the particle system \eqref{eq:part_fr_KS}. Such a result falls after getting an estimation of the expectation of some singular function of the distance between the first and second particle (by ex changeability), but this estimate is obtained with very different techniques in the Diffusion Dominated and the Fair Competition case.

\subsection{Useful estimates}
In this section we provide some basic estimates which will be used later in the paper. We begin with the
\begin{lemma}
	\label{lem:use1}
	The following useful properties hold
	\begin{itemize}
		\item[$(i)$] For all $a,b\geq 0$ and $\alpha,\beta\geq 0$. Then
		\[
		\lt( \alpha a-\beta b\rt)\lt( \ln a- \ln b \rt)\geq (a-b)(\alpha-\beta),
		\]
		with equality if and only if $a=b$.
		\item[$(ii)$] the functional $\Phi$ defined on $(0,\infty)\times(0,\infty)$ as $\Phi(x,y)=(x-y)(\ln x - \ln y)$ is convex and affine along the lines passing through the origin.
		\item[$(iii)$] for any $\kappa\geq 1$ and $x,y\in \R^2$ it holds
		\[
		\lt(\lal x\ral^{\kappa-2}x-\lal y\ral^{\kappa-2}y \rt)\cdot (x-y) \geq 0. 
		\]
	\end{itemize}
\end{lemma}

\begin{proof}
	$(i)$  By symmetry of the role of $a,b,\alpha,\beta$ we only treat the cases $a\geq b, \alpha\geq \beta$ and $a\geq b, \alpha\leq\beta$. In the first case we easily obtain, since $x\mapsto \ln x$ is increasing
	\[
	\begin{split}
	(a\alpha-b\beta)(\ln a- \ln b&)= (\alpha-\beta)a(\ln a-\ln b)+\beta (a-b)(\ln a -\ln b)\\
	&\quad \geq (\alpha-\beta )(a-b),	
	\end{split}
	\]
	because since $x\mapsto \ln x$ is concave, it holds 
	\[
	(\ln a- \ln b)\geq\frac{(a-b)}{a},
	\]
	with equality if and only if $a=b$.	On the other case we get
	\[
	\begin{split}
	(a\alpha-b\beta)(\ln a- \ln b&)= \alpha(a-b)(\ln a-\ln b)+b(\ln a -\ln b) (\alpha-\beta)\\
	&\quad \geq (\alpha-\beta )(a-b),	
	\end{split}
	\]
	because it holds
	\[
	(\ln a- \ln b)\leq\frac{(a-b)}{b},
	\]
	which concludes the proof.\newline
	$(ii)$ direct computations yields
	\[
	\nabla^2\Phi(x,y)=(x+y)\begin{pmatrix}
	\frac{1}{x^2} & -\frac{1}{xy} \\ 
	-\frac{1}{xy} & \frac{1}{y^2}
	\end{pmatrix},
	\]
	which is nonnegative. Moreover it is clear that for any $k>0$
	\[
	\Phi(kx,ky)=(kx-ky)(\ln(kx)-\ln(ky))=k\Phi(x,y).
	\]
	$(iii)$ The functional $m_\kappa:x\in \R^2 \mapsto \lal x \ral^\kappa$ is convex for $\kappa\geq 1$. And then for any $x,y\in \R^2$
	\[
	\lt( \nabla m_\kappa(x)-\nabla m_\kappa(y) \rt)\cdot (x-y)\geq 0.
	\]

\end{proof}

\subsection{Fractional Laplacian and fractional Fisher information}
In this section we provide the key arguments from which the tightness of the law of the particle system will fall. We start with the tool needed in the Fair Competition case. It consists roughly in giving a bound from below of the Ito's correction of the process $(|Z_t|^\e)_{t\geq 0}$ with $\e \in (0,1)$ and $(Z_t)_{t\geq 0}$ some $2d$ and $a$-stable L\'evy process. Precisely we have the

\begin{lemma}
	\label{lem:frlap} Let $\eta>0$, $\e\in(0,1)$ and $\phi_{\eta}$ defined as
	\[
	\phi_{\eta}:x\in \R^2\mapsto \lt(|x|^2+\eta^2\rt)^{\frac{\e}{2}}=\eta^\e \lal\frac{x}{\eta} \ral^\e.
	\]
	Then for any $a\in (1,2)$ there exists a constant $C_{\e,a}$ such that for any $x\in \R^2$ it holds
	
	\[
	\begin{aligned}
	\int_{\R^2}\frac{\phi_{\eta}(x+z)-\phi_{\eta}(x)-z\cdot \nabla \phi_{\eta}(x)}{|z|^{2+a}}dz&\geq \frac{\e^2\pi}{2(2-a)}(|x|^2+\eta^2)^{\frac{\e-4}{2}} |x|^{4-a}-\e2\pi C_{\e,a}|x|^{\e-a},
	\end{aligned}
	\]
	with 
	\[
	C_{\e,a}=\lt(\frac{2-\e}{\sqrt{4-\e}(3-a)\e}+\frac{1 }{a\e}\rt).
	\]
	
\end{lemma}

\begin{proof}
	Let be $x,z\in \R^2\times \R^2$, $\eta>0$ and define $\phi_{\eta}^{z,x}$ as
	\[
	\phi_{\eta}^{z,x}(t)=\lt( |x+tz|^2+\eta^2\rt)^{\frac{\e}{2}}.
	\]
	Then straightforward computations yield
	\[
	\begin{split}
	&(\phi_{\eta}^{z,x})'(t)=\e\left\langle z,x+tz \right\rangle\lt( |x+tz|^2+\eta^2\rt)^{\frac{\e-2}{2}},\\
	&(\phi_{\eta}^{z,x})''(t)=\e\lt( |x+tz|^2+\eta^2\rt)^{\frac{\e-2}{2}}|z|^2\lt( (\e-2)X_{x,t,z}^2+1    \rt),\\
	&(\phi_{\eta}^{z,x})^{(3)} (t)=\e(\e-2)\lt( |x+tz|^2+\eta^2\rt)^{\frac{\e-3}{2}}|z|^3\lt((\e-4)X_{x,t,z}^3+3X_{x,t,z}	\rt),
	\end{split}
	\] 
	with
	$X_{x,t,z}=\left\langle \frac{x+tz}{|x+tz|},\frac{z}{|z|} \right\rangle\lt(\frac{|x+tz|^2}{|x+tz|^2+\eta^2}\rt)^{\frac{1}{2}}$. Then Taylor's formula yields
	\[
	\begin{aligned}
	\int_{\R^2}\frac{\phi_{\eta}^{z,x}(1)-\phi_{\eta}^{z,x}(0)-(\phi_{\eta}^{z,x})'(0)}{|z|^{2+a}}dz&= \int_{ |z|\leq |x|} \frac{\phi_{\eta}^{z,x}(1)-\phi_{\eta}^{z,x}(0)-(\phi_{\eta}^{z,x})'(0)}{|z|^{2+a}}\\
	&\quad +   \int_{|z|\geq |x|}\frac{\phi_{\eta}^{z,x}(1)-\phi_{\eta}^{z,x}(0)}{|z|^{2+a}}dz\\
	&\quad =\int_{|z|\leq |x|}\frac{(\phi_{\eta}^{z,x})''(0)}{2|z|^{2+a}}dz\\
	&\quad +\int_{|z|\leq |x|} \lt(\int_0^1 \frac{(\phi_{\eta}^{z,x})^{(3)}(t)}{2!}(1-t)^2dt \rt) \frac{dz}{|z|^{2+a}}\\
	&\quad +  \int_{|z|\geq |x|}\frac{\phi_{\eta}^{z,x}(1)-\phi_{\eta}^{z,x}(0)}{|z|^{2+a}}dz\\
	&\quad :=I_1+I_2+I_3.
	\end{aligned}
	\]
	$\diamond$ Estimate of $I_1$\newline
	It is direct to obtain rewriting $x=|x|(\cos(\theta_0),\sin(\theta_0))$
	\[
	\begin{aligned}
	I_1&=\frac{\e}{2}\lt( |x|^2+\eta^2\rt)^{\frac{\e-2}{2}}\int_{|z|\leq |x|}|z|^{-a}\lt( (\e-2)\left\langle \frac{z}{|z|},\frac{x}{|x|} \right\rangle^2\frac{|x|^2}{|x|^2+\eta^2}+1    \rt)dz\\
	&\quad=
	\frac{\e}{2} \lt(|x|^2+\eta^2\rt)^{\frac{\e-2}{2}}\int_0^{|x|}\int_0^{2\pi}r^{-a}\lt(1+(\e-2)(\cos(\theta_0)\cos(\theta)+\sin(\theta_0)\sin(\theta))^2\frac{|x|^2}{|x|^2+\eta^2} \rt)rdrd\theta\\
	&\quad = \frac{\e}{2} \lt(|x|^2+\eta^2\rt)^{\frac{\e-2}{2}}\lt(\int_0^{|x|}r^{	1-a}dr\rt)\lt(\int_0^{2\pi}\lt(1+(\e-2)\cos^2(\theta_0-\theta)\frac{|x|^2}{|x|^2+\eta^2}\rt)d\theta\rt)\\
	&\quad = \frac{\e}{2} \lt(|x|^2+\eta^2\rt)^{\frac{\e-2}{2}}\frac{|x|^{2-a}}{2-a}2\pi\lt( 1+\frac{\e-2}{2} \frac{|x|^2}{|x|^2+\eta^2}\rt)\\
	&\quad =2\pi\lt( \frac{\e^2}{4(2-a)}\lt(|x|^2+\eta^2\rt)^{\frac{\e-4}{2}}|x|^{4-a}+\frac{\e}{2(2-a)}\lt(|x|^2+\eta^2\rt)^{\frac{\e-2}{2}}|x|^{2-a}\frac{\eta^2}{|x|^2+\eta^2}\rt)\\
	&\quad \geq \frac{\pi\e^2}{2(2-a)}\lt(|x|^2+\eta^2\rt)^{\frac{\e-4}{2}}|x|^{4-a}.
	\end{aligned}
	\]

	$\diamond$ Estimate of $I_2$ \newline
	In the case $|z|\leq |x|$ we easily get
	\[
	|x+tz|\geq |x|-t|z|\geq (1-t)|x|,
	\]
	and since 
	\[
	\lt| (\e-4)X^3+3X \rt|\leq \frac{2}{\sqrt{4-\e}},
	\]
	for $X\in (-1,1)$ we deduce
	\[
	|(\phi_{\eta}^{z,x})^{(3)} (t)|\leq \e(2-\e)|x|^{\e-3}|z|^3(1-t)^{\e-3} \frac{2}{\sqrt{4-\e}}, 
	\]
	therefore
	\[
	\begin{aligned}
	I_2\geq& -\frac{\e(2-\e)}{2} \frac{2}{\sqrt{4-\e}}\lt(\int_0^1(1-t)^{\e-1} dt \rt) |x|^{\e-3} \int_{|z|\leq |x|}\frac{|z|^3}{|z|^{2+a}}dz\\
	&\quad \geq  - \frac{2\pi\e(2-\e)}{\e(3-a)\sqrt{4-\e}}|x|^{\e-a}.
	\end{aligned}
	\]

	$\diamond$ Estimate of $I_3$ \newline
	Note that if $|z|\geq |x|$ then 
	\[
	-(|z|^2+\eta^2)^{\frac{\e}{2}}=-\phi^{x,z}_{\eta}(0)\leq -\phi^{z,x}_{\eta}(0)=-(|x|^2+\eta^2)^{\frac{\e}{2}},
	\]
	then since $\phi_{\eta}^{z,x}(1)=\phi_{\eta}^{x,z}(1)$, we have
	\[
	\begin{aligned}
	I_3=& \int_{|z|\geq |x|}\frac{\phi_{\eta}^{z,x}(1)-\phi_{\eta}^{z,x}(0)}{|z|^{2+a}}dz\\
	&\quad \geq  \int_{|z|\geq |x|}\frac{\phi_{\eta}^{x,z}(1)-\phi_{\eta}^{x,z}(0)}{|z|^{2+a}}dz\\
	&\quad =\int_{|z|\geq |x|}\frac{(\phi_{\eta}^{x,z})'(0)+\int_0^1(\phi_{\eta}^{x,z})^{(2)}(t)(1-t)dt}{|z|^{2+a}}dz\\
	&\quad =\int_{|z|\geq |x|}\frac{\e (|z|^2+\eta^2)^{\frac{\e-2}{2}}\left\langle x,z \right\rangle }{|z|^{2+a}}dz + \int_{|z|\geq |x|} \int_0^1(\phi_{\eta}^{x,z})^{(2)}(t)(1-t) |z|^{-(2+a)}dtdz\\
	&\quad = \int_{|z|\geq |x|} \int_0^1(\phi_{\eta}^{x,z})^{(2)}(t)(1-t) |z|^{-(2+a)}dtdz,
	\end{aligned}
	\]
	but recall that
	\[
	(\phi_{\eta}^{x,z})''(t)=\e\lt( |z+tx|^2+\eta^2\rt)^{\frac{\e-2}{2}}|x|^2\lt( (\e-2)\left\langle \frac{x}{|x|},\frac{z+tx}{|z+tx|} \right\rangle^2\frac{|z+tx|^2}{|z+tx|^2+\eta^2}+1    \rt),
	\]
	and since $|z|\geq |x|$ we easily get
	\[
	|z+tx|\geq |z|-t|x|\geq (1-t)|x|.
	\]
	Moreover since for any $X\in(0,1)$ 
	\[
	\lt|(\e-2)X^2+1\rt|\leq 1,
	\]
	it follows
	\[
	|(\phi_{\eta}^{z,x})''(t)|\leq \e|x|^\e(1-t)^{\e-2}.
	\]
	This yields
	\[
	I_3\geq - \frac{2\pi\e}{\e a   }|x|^{\e-a},
	\]
	and the result holds with the desired constant.
\end{proof}

\begin{remark}
	 Also recall the bound for any $\e\in (0,a)$ 
	\bq
	\label{eq:momfralap}
	-(-\Delta)^{a/2}(m_\e)(x)\leq C_{a,\e} m_{\e-a}(x),
	\eq
	see for instance \cite[Proposition 2.2, (18)]{LL}. 
\end{remark}

Next we need some tools for the Diffusion Dominated case $\alpha<a$. As this case is close to the one studied in \cite{GQ} which relies on properties of the classical Fisher information, we need to extend those properties to the anomalous diffusion case. Such a fractional Fisher information has not been very studied in the literature. The main results in this domain, to the best of the author's knowledge, have been obtained by Toscani \cite{Tos}, also note the contribution of \cite{Gran} where the author also consider such a fractional Fisher information for probability measures on the $1d$ torus or the real line. The work by Erbar in \cite{Erb} should also be quoted, where the author establishes some new metric on the probability measure space, with respect to which the Boltzmann's entropy is a gradient flow functional for the fractional heat equation. In this purpose the author introduces a (relative) fractional Fisher information.  \newline
By definition, considering $N$ independent $2$-dimensional Brownian motion and one $2N$-dimensional Brownian motion is the same. Therefore establishing the Liouville equation associated to a particle system with independent Brownian diffusions, falls by using an Ito's formula in $\R^{2N}$ (see for instance \cite[Proof of Proposition 5.1, Step 1]{MSN}). However, for any $a\in (0,2)$, $N$ independent $a$-stable L\'evy process on $\R^2$ is L\'evy processes on $\R^{2N}$ which is not a $a$-stable. \newline
Let be $\mu_{1},\cdots,\mu_N$ $N$ independent  Poisson random measures on $\R^+\times \R^2$ with intensity $ds\times c_{2,a}\frac{dx}{|x|^{2+a}}$, and denote $\tilde{\ZZ}_t^{N}$ the $\R^{2N}$-valued process defined by
\[
\tilde{\ZZ}_t^{N}=\lt( \int_{[0,t]\times\R^2} x\bar{\mu}_1(ds,dx) ,\cdots,\int_{[0,t]\times\R^2} x\bar{\mu}_N(ds,dx) \rt ).
\]
It is classical (see \cite[Chapter VII, 3)]{Cin}), since the $(\mu_i)_{i=1,\cdots,N}$ are independent, to obtain for any $r^N=(r_1,\cdots,r_N)\in\R^{2N}$
\[
\E\lt[e^{ir^N\cdot \tilde{\ZZ}_t^{N} } \rt]=e^{t \sum_{k=1}^N \comment{c_{a,2}} |r_k|^a} =e^{t\sum_{k=1}^{N}\psi(r_k)},
\]
where $\psi(r)=\comment{c_{a,2}}|r|^a$ is the characteristic exponent of the $2d$, $a$-stable L\'evy process. Let now be $\mathcal{M}^{N}$ be a Poisson random measure on $\R^+\times \R^2\times\{1\cdots,N\}$ with intensity $Nds\times c_{2,a} |x|^{-(2+a)}dx\times \mathcal{U}\lt(\{1,\cdots,N\}\rt)$ 
and denote
\[
\bar{\ZZ}_t^{N}= \int_{[0,t]\times \R^2\times\{1\cdots,N\}}\lt( \delta_{l=1}x,\cdots,\delta_{l=N}x \rt)\bar{\mathcal{M}^{N}}(ds,dx,dl).
\]
Hence it holds due to classical properties of Poisson random measure (see \cite[Theorem 2.9, p 252]{Cin} for instance)
\[
\begin{split}
\E\lt[e^{ir^N\cdot \bar{\ZZ}_t^{N} } \rt]& =   e^{-\int_0^t c_{2,a}\mbox{v.p.}\int_{ \R^{2} } \frac{1}{N}\sum_{l=1}^N\lt(  1-e^{i\sum_{k=1}^N \delta_{l,k}x\cdot r_k} \rt)  N \frac{dx}{|x|^{2+a}}ds} \\
&\quad =  e^{t \sum_{k=1}^N c_{2,a} \mbox{v.p.}\int_{ \R^{2} }\lt(  e^{ix\cdot r_k}-1 \rt) \frac{dx}{|x|^{2+a}} dx }=e^{t\sum_{k=1}^N\psi(r_k)}.
\end{split}
\]
Then it deduces that the law $F_t^N=\LL(\tilde{\ZZ}_t^N)=\LL(\bar{\ZZ}_t^N)$ solves the following linear integro-differential PDE
\bq
\pa_t F_t^N(X^N)+\int_{\R^{2N}} \frac{F_t^N(X^N+z\mathbf{a_k})-F_t^N(X^N))}{|z|^{2+a}}dz=0,
\eq
with $z\mathbf{a}_k=\lt(\underbrace{0,\cdots 0}_{2(k-1)} ,
z_1,z_2,\underbrace{0\cdots,0}_{2(N-k)}\rt) \in \R^{2N}$ and $z=(z_1,z_2)$. Next we look at the dissipation of entropy along this equation. For a probability measure $F^N\in \PP(\R^{2N})$ introduce the \textit{normalized Boltzmann's entropy} 
\[
\HH^N(F^N)=\frac{1}{N}\int_{ \R^{2N} } F^N(x)\ln F^N(x)dx,
\]
then it holds
\[
\begin{split}
\frac{d}{dt}\HH^N(F^N_t)&=-\frac{1}{N}\int_{ \R^{2N} } (1+\ln F_t^N(X^N) )\lt(\sum_{k=1}^N\mbox{p.v.}\int_{\R^2} \frac{ F_t^N(X^N+z\mathbf{a}_k)-F_t^N(X^N)  }{|z|^{2+a}}dz\rt)dX^N\\
&\quad = -\frac{1}{N} \sum_{k=1}^N\int_{ \R^{2N} }  \mbox{p.v.}\int_{\R^2} \frac{\lt( F_t^N(X^N+z\mathbf{a}_k)-F_t^N(X^N) \rt)\ln F_t^N(X^N) }{|z|^{2+a}}dzdX^N\\
&\quad = \frac{1}{N} \sum_{k=1}^N\int_{ \R^{2N} }  \mbox{p.v.}\int_{\R^2} \frac{\lt( F_t^N(X^N)-F_t^N(X^N+z\mathbf{a}_k) \rt)\ln F_t^N(X^N+z\mathbf{a}_k) }{|z|^{2+a}}dzdX^N\\
&\quad =-\frac{1}{2}\frac{1}{N} \sum_{k=1}^N\int_{ \R^{2N} }  \mbox{p.v.}\int_{\R^2} \frac{\lt( F_t^N(X^N+z\mathbf{a}_k)-F_t^N(X^N) \rt)\lt(\ln F_t^N(X^N+z\mathbf{a}_k)-\ln F_t^N(X^N) \rt) }{|z|^{2+a}}dzdX^N\\
&\quad = -\frac{1}{2}\frac{1}{N} \sum_{k=1}^N\int_{ \R^{2N} }  \mbox{p.v.}\int_{\R^2} \frac{\Phi\lt( F_t^N(X^N+z\mathbf{a}_k),F_t^N(X^N) \rt) }{|z|^{2+a}}dzdX^N.
\end{split}
\]
Henceforth we are inclined to consider the functional defined on $\PP(\R^{2N})$ as
\bq
\label{eq:FFI}
\begin{split}
\II^N_a(G_N)&=\frac{1}{2}\frac{1}{N} \sum_{k=1}^N\int_{ \R^{2N} }  \mbox{p.v.}\int_{\R^2} \frac{\lt( G_N(x+z\mathbf{a}_k)-G_N(x) \rt)\lt(\ln G_N(x+z\mathbf{a}_k)-\ln G_N(x) \rt) }{|z|^{2+a}}dzdx\\
&\quad =\frac{1}{2}\frac{1}{N} \sum_{k=1}^N \int_{ \R^{2(N-1)} } \int_{  \R^4 }  \frac{\lt( G_N(X^x_k)-G_N(X_k^y) \rt)\lt(\ln G_N(X_k^x)-\ln G_N(X_k^y) \rt) }{|x-y|^{2+a}}dxdy\,dx_1,\cdots,dx_{k-1},dx_{k+1},\cdots,dx_N,\\  
&\quad \mbox{with} \quad X_k^x=(x_1,\cdots,x_{k-1},x,x_{k+1},\cdots,x_N)\\
&\quad =\frac{1}{2}\frac{1}{N} \sum_{k=1}^N \int_{ \R^{2(N-1)} } \int_{  \R^2\times \R^2 }  \frac{\Phi\lt( G_N(X^x_k),G_N(X_k^y) \rt) }{|x-y|^{2+a}}dxdy\,dx_1,\cdots,dx_{k-1},dx_{k+1},\cdots,dx_N,
\end{split}
\eq
as the pendent for fractional diffusion of the normalized Fisher information (with the convention $\II_a^1=I_a$).
\begin{remark}
	In the classical case the Fisher information can be rewrtiten as
	\[
	\begin{aligned}
	\II^N(F^N)&=\frac{1}{N}\int_{ \R^{2N} } \frac{\lt|\nabla F^N\rt|^2}{F^N} dX^N = \frac{1}{N}\sum_{i=1}^N\int_{ \R^{2N} }   \frac{\lt|\nabla_i F^N\rt|^2}{F^N}dX^N\\
	&\quad = \frac{1}{N}\sum_{i=1}^N\int_{ \R^{2N} }  \nabla_i F^N\cdot \nabla_i \ln F^NdX^N,
	\end{aligned}
	\]	
	which has the same the form of the one defined in \eqref{eq:FFI}, except that the $H^1$ inner product between $F^N$ and $\ln F^N$ w.r.t. the $i$-th component is replaced in the fractional case with the $H^{a/2}$ inner product. 
\end{remark}
 This quantity is so far an entropy dissipation, but not an \textit{Information} yet. In oder to properly qualify it as such, we have the

\begin{proposition}
	\label{prop:FFI}
	The fractional Fisher information defined in \eqref{eq:FFI}
	\begin{itemize}
		\item[$(i)$] is proper, convex, 
		\item[$(ii)$] is lower semi continuous w.r.t. the weak convergence in $\PP(\R^{2N})$,
		\item[$(iii)$] is super-additive in the sense that for $G^N\in \mathcal{P}(\R^{2N})$ and $G^i\in \mathcal{P}(\R^{2i}),G^{N-i}\in \PP(\R^{2(N-i)})$ its marginal on $\R^{2i}$ (resp. $\R^{2(N-i)}$ for $i=1,\cdots,N-1$ it holds
		\[
		\II^N_a(G^N)\geq \frac{i}{N}\II^i_a(G^i)+\frac{N-i}{N}\II_a^{N-i}(G^{N-i}).
		\]
		Moreover equality holds if and only $G^N=G^i\otimes G^{N-i}$. 
		\item[$(iv)$] satisfies, for any $G^N\in \PP_{sym}(\R^{2N})$, and $G^k\in \PP_{sym}(\R^{2k})$ its marginal on $\R^{2k}$
		\[
		\II_a^N(G^N)\geq \II_a^k(G^k), 
		\]  
		and for any $g\in \PP(\R^2)$
		\[
		\II_a^N(g^{\otimes N})=I_a(g).
		\]
	
	\end{itemize}

\end{proposition}

\begin{proof}
	Proof of point $(i)$\newline
	Convexity holds form point $(ii)$ of Lemma \ref{lem:use1}. We delay the proof of the fact that $\II_a^N$ is proper after the proof of point $(iv)$.
	\newline
	Proof of point $(ii)$\newline
	Let be $(u^k_n)_{n\in \mathbb{N}}$, such that $u^k_n\overset{*}{\rightharpoonup} u^k\in \PP(\R^{2k})$ and for $\e>0$ set $\rho_\e(x_1,\cdots,x_k)=\frac{1}{(2\pi\e)^{\frac{k}{2}}}e^{-\frac{\sum_{k=1}^N|x_k|^2}{2\e}}$. Then it holds that $u^k_n*\rho_\e=:u^{k,\e}_n\overset{*}{\rightharpoonup} u^{
		k,\e}:=u^k*\rho_\e\in \PP(\R^{2k})$ for any $\e>0$, and that $u^{k,\e}$ is a smooth function which is always strictly larger than $0$. Note that due to point $(ii)$ of Lemma \ref{lem:use1} and Jensen's inequality it holds
	\[
	\begin{split}
	\II_a^k(u^{k,\e}_n)&=\frac{1}{2k} \sum_{j=1}^k \int_{ \R^{2(k-1)} } \int_{  \R^2\times \R^2 }  \frac{\Phi\lt( \int u^{k}_n(X^x_j+z)\rho_\e(z)dz,\int u^{k}_n(X_j^y+z)\rho_\e(z)dz \rt) }{|x-y|^{2+a}}dxdy\,d\hat{X}^j,\\	
	&\quad \leq \frac{1}{2k} \int\sum_{j=1}^k \int_{ \R^{2(k-1)} } \int_{  \R^2\times \R^2 }  \frac{\Phi\lt(  u_n^k(X^x_j+z), u_n^k(X_j^y+z) \rt) }{|x-y|^{2+a}} dxdy\,d\hat{X}^j_N\rho_\e(z)dz=	\II_a^k(u_n^k).
	\end{split}
	\]
Moreover it holds (dealing only with the case $j=1$ by symmetry)

	\[
	\begin{aligned}
	\II^k_a(u^{k,\e}_n)=&\int_{\R^{2(k-1)}} \int_{\R^2\times \R^2} \frac{\lt(u^{k,\e}_n(x,X^k)-u^{k,\e}_n(
		y,X^k)\rt)\lt( \ln u^{k,\e}_n(x,X^k)-\ln u^{k,\e}_n(y,X^k)\rt)}{|x-y|^{2+a}}dxdydX^k\\
	&=\int_{\R^{2(k-1)}} \int_{\R^2\times \R^2} \frac{\lt(u^{k,\e}_n(x,X^k)-u^{k,\e}_n(
		y,X^k)\rt)\lt( \ln u^{k,\e}(x,X^k)-\ln u^{k,\e}(y,X^k)\rt)}{|x-y|^{2+a}}dxdydX^k\\
	&+\int_{\R^{2(k-1)}} \int_{\R^2\times \R^2} \frac{\lt(u^{k,\e}_n(x,X^k)-u^{k,\e}_n(
		y,X^k)\rt)\lt( \ln \lt(\frac{u^{k,\e}_n(x,X^k)}{u^{k,\e}(x,X^k)}\rt)-\ln\lt(\frac{u^{k,\e}_n(y,X^k)}{u^{k,\e}(y,X^k)}\rt)\rt)}{|x-y|^{2+a}}dxdydX^k\\
	&:=\mathcal{E}^n_1+\mathcal{E}^n_2.
	\end{aligned}
	\]
	$\diamond$ Estimate of $\mathcal{E}^n_1$ \newline
	By symmetry we rewrite 
	\[
	\begin{split}
		\mathcal{E}^n_1&= 2\int_{\R^{2k}} u^{k,\e}_n(x,X^k)\int_{  \R^2 } \frac{\lt( \ln u^{k,\e}(x,X^k)-\ln u^{k,\e}(y,X^k)\rt)}{|x-y|^{2+a}}dy dxdX^k.
	\end{split}
	\]
	
	$\diamond$ Estimate of $\mathcal{E}^n_2$ \newline
	Using point $(i)$ of Lemma \ref{lem:use1} and also by symmetry we find that 
	\[
	\begin{split}
	\mathcal{E}^n_2&\geq \int_{\R^{2(k-1)}} \int_{\R^2\times \R^2} \frac{\lt(u^{k,\e}(x,X^k)-u^{k,\e}(
		y,X^k)\rt)\lt(  \frac{u^{k,\e}(x,X^k)}{u^\e(x,X^k)}-\frac{u^\e_n(y,X^k)}{u^\e(y,X^k)}\rt)}{|x-y|^{2+a}}dxdydX^k\\	
	&\quad = 2\int_{\R^{2k}} u^{k,\e}_n(x,X^k)\int_{  \R^2 } \frac{\lt(u^{k,\e}(x,X^k)-u^{k,\e}(
		y,X^k)\rt)}{u^{k,\e}(x,X^k)|x-y|^{2+a}}dy dxdX^k.
	\end{split}
	\]
	Now since the functions
	\[
	(x,X_k)\mapsto \mbox{v.p.}\int_{  \R^2 } \frac{\lt( \ln u^{k,\e}(x,X^k)-\ln u^{k,\e}(y,X^k)\rt)}{|x-y|^{2+a}}dy, 
	\]
	and
	\[
	(x,X_k)\mapsto \mbox{v.p.}\int_{  \R^2 } \frac{\lt(u^{k,\e}(x,X^k)-u^{k,\e}(
		y,X^k)\rt)}{u^{k,\e}(x,X^k)|x-y|^{2+a}}dy,
	\]
	are continuous and bounded for any $\e>0$, we deduce that 
	\[
	\lim_{n\rightarrow \infty}\mathcal{E}_1^n=2\int_{\R^{2k}} u^{k,\e}(x,X^k)\int_{  \R^2 } \frac{\lt( \ln u^{k,\e}(x,X^k)-\ln u^{k,\e}(y,X^k)\rt)}{|x-y|^{2+a}}dy dxdX^k=\II_a^k(u^{k,\e}),
	\]
	and 
	\[
	\lim_{n\rightarrow \infty}\mathcal{E}_2^n=2\int_{\R^{2k}} u^{k,\e}(x,X^k)\int_{  \R^2 } \frac{\lt(u^{k,\e}(x,X^k)-u^{k,\e}(
	y,X^k)\rt)}{u^{k,\e}(x,X^k)|x-y|^{2+a}}dy dxdX^k=0.
	\]
	Hence it deduces that 
	\[
	\liminf_{n}\II_a^k(u^k_n)\geq \liminf_{n}\II_a^k(u^{k,\e}_n)= \II_a^k(u^{k,\e}).
	\]
	Then for almost all $x,y,x_2,\cdots,x_k\in\R^{2(k+1)}$ it holds 
	\[
	\lim_{\e\rightarrow 0}\frac{\Phi(u^{k,\e}(x,X^k),u^{k,\e}(y,X^k)) }{|x-y|^{2+a}}=\frac{\Phi(u^{k}(x,X^k),u^{k}(y,X^k)) }{|x-y|^{2+a}},
	\]
	up to taking a sequence $(\e_n)_{n\geq 1}$ converging to $0$. Hence by Fatou's Lemma it holds 
	\[
	\begin{aligned}
	\II_a^k(u^{k})&=\int_{\R^{2(k-1)}}\int_{\R^2\times\R^2}\frac{\Phi(u^{k}(x,X^y),u^{k}(y,X^k)) }{|x-y|^{2+a}}dxdydx_2\cdots dx_k\\
	&\quad \leq \liminf_{\e>0}\int_{\R^{2(k-1)}}\int_{\R^2\times\R^2}\frac{\Phi(u^{k,\e}(x,X^k),u^{k,\e}(y,X^k)) }{|x-y|^{2+a}}dxdydx_2\cdots dx_k=\liminf_{\e>0}\II_a^k(u^{k,\e}),
	\end{aligned}
	\]
	and therefore
	\[
	\liminf_{n}\II_a^k(u^k_n)\geq \II_a^k(u^{k}).
	\]
	
	 .\newline
	Proof of point $(iii)$\newline
	Let be $G_N\in \PP(\R^{2N})$ and $(X_1,\cdots,X_N)$ a random vector of law $G_N$, fix $i=1,\cdots,N$ and denote 
	\[
	g_{i}(x_1,\cdots,x_i|x_{i+1},\cdots,x_N)=\LL(X_1,\cdots,X_i\, | \, X_{i+1},\cdots,X_N),
	\]
	then 
	\[
	G_N(x_1,\cdots,x_N)=g_i(x_1,\cdots,x_i|x_{i+1},\cdots,x_N) G_{N-i}(x_{i+1},\cdots,x_N)=g_{N-i}(x_{i+1},\cdots,x_N|x_{1},\cdots,x_i)G_{i}(x_{1},\cdots,x_i).
	\]
	Next observe that for $k\leq i$
	\[
	\begin{split}
	\Phi\lt( G_{N}(X^x_k), G_{N}(X_k^y)\rt)&=\lt( G_{N}(X^x_k)- G_{N}(X_k^y)\rt)\lt( \ln G_{i}(X^x_{k,i})-\ln G_{i}(X^y_{k,i})  \rt)\\
	&\quad + \lt( G_{N}(X^x_k)- G_{N}(X_k^y)\rt)\lt( \ln g_{N-i}(X^{N-i}|X^x_{k,i})-\ln g_{N-i}(X^{N-i}|X^y_{k,i})  \rt),
	\end{split}
	\]
	with the notations
	\[
	\begin{split}
	&X^x_k=(x_1,\cdots,x_{k-1},x,x_{k+1},\cdots,x_N),\\
	&X^x_{k,i}= (x_1,\cdots,x_{k-1},x,x_{k+1},\cdots,x_i), \quad \mbox{if} \quad k=i \, \mbox{ the last component is}\, x\\
	&X^{N-i}= (x_{i+1},\cdots,x_N).
	\end{split}
	\]
	Similarly if $k>i$
	\[
	\begin{split}
	\Phi\lt( G_{N}(X^x_k), G_{N}(X_k^y)\rt)&=\lt( G_{N}(X^x_k)- G_{N}(X_k^y)\rt)\lt( \ln G_{N-i}(X^x_{k,N-i})-\ln G_{N-i}(X^y_{k,N-i})  \rt)\\
	&\quad + \lt( G_{N}(X^x_k)- G_{N}(X_k^y)\rt)\lt( \ln g_{i}(X^{i}|X^x_{k,N-i})-\ln g_{i}(X^{i}|X^y_{k,N-i})  \rt),
	\end{split}
	\]
	with the similar notations
	\[
	\begin{split}
	&X^x_{k,N-i}= (x_{i+1},\cdots,x_{k-1},x,x_{k+1},\cdots,x_N), \quad \mbox{if} \quad k=i+1 \, \mbox{ the first component is}\, x\\
	&X^{i}= (x_{1},\cdots,x_i).
	\end{split}
	\]
	Using all this yields

	\[
	\begin{split}
	\II_a^k(G_{N})&=\frac{1}{2}\frac{1}{N} \sum_{k=1}^N \int_{ \R^{2(N-1)} } \int_{  \R^2\times \R^2 }  \frac{\Phi\lt(  G_N(X^x_k), G_N(X_k^y) \rt) }{|x-y|^{2+a}}dxdy\,d\hat{X}^k_N,\\	
	&\quad \geq \frac{1}{2}\frac{1}{N} \int\sum_{k=1}^i \int_{ \R^{2(N-1)} } \int_{  \R^2\times \R^2 }  \frac{	\lt( G_{N}(X^x_k)- G_{N}(X_k^y)\rt)\lt( \ln G_{i}(X^x_k)-\ln G_{i}(X^y_k)  \rt) }{|x-y|^{2+a}} dxdy\,d\hat{X}^k_N\\
	&\quad +\frac{1}{2}\frac{1}{N} \int\sum_{k=1}^i \int_{ \R^{2(N-1)} } \int_{  \R^2\times \R^2 }  \frac{\lt( G_{N}(X^x_k)- G_{N}(X_k^y)\rt)\lt( \ln g_{N-i}(X^{N_i}|X^x_{k,i})-\ln g_{N-i}(X^{N-i}|X^y_{k,i})  \rt) }{|x-y|^{2+a}} dxdy\,d\hat{X}^k_N\\
	&\quad + \frac{1}{2}\frac{1}{N} \int\sum_{k=i+1}^N \int_{ \R^{2(N-1)} } \int_{  \R^2\times \R^2 }  \frac{\lt( G_{N}(X^x_k)- G_{N}(X_k^y)\rt)\lt( \ln G_{N-i}(X^x_{k,N-i})-\ln G_{N-i}(X^y_{k,N-i})  \rt) }{|x-y|^{2+a}} dxdy\,d\hat{X}^k_N\\
	&\quad + \frac{1}{2}\frac{1}{N} \int\sum_{k=i+1}^N \int_{ \R^{2(N-1)} } \int_{  \R^2\times \R^2 }  \frac{\lt( G_{N}(X^x_k)- G_{N}(X_k^y)\rt)\lt( \ln g_{i}(X^{i}|X^x_{k,N-i})-\ln g_{i}(X^{i}|X^y_{k,N-i})  \rt) }{|x-y|^{2+a}} dxdy\,d\hat{X}^k_N\\
	&\quad := \mathcal{J}^i_1+\mathcal{J}^i_2+\mathcal{J}^{N-i}_1+\mathcal{J}^{N-i}_2.
	\end{split}
	\]
	$\diamond$ Estimate of $\mathcal{J}_1$:\newline
	Both terms are treated equally, so we will focus only on the $i$ term. Using Fubini's Theorem yields 
	\[
	\begin{aligned}
	\mathcal{J}^1_1&=\frac{1}{2}\frac{1}{N} \int\sum_{k=1}^i \int_{ \R^{2(N-1)} } \int_{  \R^2\times \R^2 }  \frac{	G_{N}(X^x_{k,i},X^{N-i})\lt( \ln G_{i}(X^x_{k,i})-\ln G_{i}(X^y_{k,i})  \rt) }{|x-y|^{2+a}} dxdy\,d\hat{X}^k_N\\
	&\quad -\frac{1}{2}\frac{1}{N} \int\sum_{k=1}^i \int_{ \R^{2(N-1)} } \int_{  \R^2\times \R^2 }  \frac{	G_{N}(X^y_{k,i},X^{N-i})\lt( \ln G_{i}(X^x_{k,i})-\ln G_{i}(X^y_{k,i})  \rt) }{|x-y|^{2+a}} dxdy\,d\hat{X}^k_N\\
	&\quad =\frac{1}{2}\frac{1}{N} \int\sum_{k=1}^i \int_{ \R^{2(i-1)} } \int_{  \R^2\times \R^2 }  \frac{	\int G_{N}(X^x_{k,i},X^{N-i}) dX^{N-i}\lt( \ln G_{i}(X^x_{k,i})-\ln G_{i}(X^y_{k,i})  \rt) }{|x-y|^{2+a}} dxdy\,d\hat{X}^k_{i}\\
	&\quad -\frac{1}{2}\frac{1}{N} \int\sum_{k=1}^i \int_{ \R^{2(i-1)} } \int_{  \R^2\times \R^2 }  \frac{	\int G_{N}(X^y_{k,i},X^{N-i}) dX^{N-i} \lt( \ln G_{i}(X^x_{k,i})-\ln G_{i}(X^y_{k,i})  \rt) }{|x-y|^{2+a}} dxdy\,d\hat{X}^k_{i}\\
	&\quad =\frac{1}{2}\frac{1}{N} \int\sum_{k=1}^i \int_{ \R^{2(i-1)} } \int_{  \R^2\times \R^2 }  \frac{	\lt( G_{i}(X^x_{k,i})- G_{i}(X^y_{k,i})\rt) \lt( \ln G_{i}(X^x_{k,i})-\ln G_{i}(X^y_{k,i})  \rt) }{|x-y|^{2+a}} dxdy\,d\hat{X}^k_{i}= \frac{i}{N} \II^i_a(G_i).
	\end{aligned}
	\]
	
	$\diamond$ Estimate of $\mathcal{J}_2$:\newline
	Similarly we only treat $\mathcal{J}^i_2$. Using point $(i)$ of Lemma \ref{lem:use1} and once again Fubini's Theorem we get
	\[
	\begin{split}
	\mathcal{J}^i_2&\geq\frac{1}{2}\frac{1}{N} \int\sum_{k=1}^i \int_{ \R^{2(N-1)} } \int_{  \R^2\times \R^2 }  \frac{\lt( G_{i}(X^x_{k,i})- G_{i}(X_{k,i}^y)\rt)\lt( g_{N-i}(X^{N-i}|X^x_{k,i})- g_{N-i}(X^{N-i}|X^y_{k,i})   \rt) }{|x-y|^{2+a}} dxdy\,d\hat{X}^k_N\\
	&= \frac{1}{2N} \int\sum_{k=1}^i \int_{ \R^{2(i-1)} } \int_{  \R^2\times \R^2 }  \frac{\lt( G_{i}(X^x_{k,i})- G_{i}(X_{k,i}^y)\rt)\lt( \int g_{N-i}(X^{N-i}|X^x_{k,i})dX^{N-i}- \int g_{N-i}(X^{N-i}|X^y_{k,i})dX^{N-i}   \rt) }{|x-y|^{2+a}} dxdy\,d\hat{X}^k_{i}\\
	&=0
	\end{split}
	\]
	
	Moreover due to point $(i)$ of Lemma \ref{lem:use1}, $\mathcal{J}_2=0$ only if for any $k=1,\cdots, i$ for almost every $x,y\in \R^2\times \R^2$ and $x_1,\cdots,x_{k-1},x_{k+1},x_i\in \R^{2(i-1)}$ it holds
	\[
	g_{N-i}(X^{N-i}|x_1,\cdots,x_{k-1},x,x_{k+1},x_i)= g_{N-i}(X^{N-i}|x_1,\cdots,x_{k-1},y,x_{k+1},x_i)=\mu(X^{N-i}),
	\] 
	for some $\mu\in \PP(\R^{2(N-i)})$. But necessarily $\mu=G_{N-i}$  and we deduce that
	\[
	G_N=G_i\otimes G_{N-i}.
	\]
	Proof of point $(iv)$\newline
	Note that the symmetry of $G_N$ yields
	\[
	\II_a^N(G_{N})=\frac{1}{2} \int_{ \R^{2(N-1)} } \int_{  \R^2\times \R^2 }  \frac{\Phi\lt(  G_N(X^x_1), G_N(X_1^y) \rt) }{|x-y|^{2+a}}dxdy\,d\hat{X}^1_N.
	\]
	In the tensorised case, Fubini's Theorem yields	
	\[
	\begin{aligned}
	\II^N_a(g^{\otimes N})=&\int \int_{\R^2\times \R^2} \frac{\lt(g(x)\prod_{k=2}^Ng(x_k)-g(y)\prod_{k=2}^Ng(x_k)\rt)\lt( \ln \lt(g(x)\prod_{k=2}^Ng(x_k)\rt)-\ln \lt(g(y)\prod_{k=2}^Ng(x_k)\rt)\rt)}{|x-y|^{2+a}}dxdy\hat{X}^1_N\\
	&=\int\lt( \int_{\R^2\times \R^2}  \frac{\lt(g(x)-g(y)\rt)\lt( \ln (g(x))-\ln (g(y))\rt)}{|x-y|^{2+a}}dxdy\rt) \prod_{k=2}^Ng(x_k)d\hat{X}^1_N= I_a(g).
	\end{aligned}
	\]
	On the other hand, with similar notations as in point $(iii)$ write
	\[
	G_N(X^x_1)=G_k(X^x_1)g_{N-k}(X^{N-k}|X^x_1),
	\]
	and then
	\[
	\begin{split}
	\Phi\lt(  G_N(X^x_1), G_N(X_1^y) \rt)&=\lt( G_{N}(X^x_1)- G_{N}(X_1^y)\rt)\lt( \ln G_{k}(X^x_1)-\ln G_{k}(X^y_1)  \rt)\\
	&\quad + \lt( G_{N}(X^x_1)- G_{N}(X_1^y)\rt)\lt( \ln g_{N-k}(X^{N-k}|X^x_1)- \ln g_{N-k}(X^{N-k}|X^y_1) \rt)\\
	&\quad \geq 	 \lt( G_{N}(X^x_1)- G_{N}(X_1^y)\rt)\lt( \ln G_{k}(X^x_1)-\ln G_{k}(X^y_1)  \rt)\\
	&\quad + \lt( G_{k}(X^x_1)- G_{k}(X_1^y)\rt)\lt( g_{N-k}(X^{N-k}|X^x_1)-  g_{N-k}(X^{N-k}|X^y_1) \rt),
	\end{split}
	\]
	dividing the above inequality by $|x-y|^{2+a}$ and integrating over $dxdydx_2,\cdots,dx_N$ yields the desired result thanks to similar computations as the one done in the proof of point $(iii)$.\newline
	To see that $\II_a^N$ is proper, take $\nu\in \PP(\R^2)$ and $\psi_{\e}=\e^{-2}e^{-\sqrt{1+\e^{-2}|x|^2}}$, and define $\nu^{N,\e}:=(\nu_\e)^{\otimes}:=(\nu*\psi_{\e})^{\otimes N}\in \PP(\R^{2N})$. Then we have
	\[
	\begin{split}
	\II_a^N(\nu^{N,\e})=&I_a(\nu*\psi_\e)=\int_{\R^4}\frac{\Phi(\nu_\e(x),\nu_\e(y))}{|x-y|^{2+a}}dxdy\\
	&\quad \leq \int_{|x-y|\leq 1} \frac{\nu_\e(x)-\nu_\e(y)}{\ln (\nu_\e(x))-\ln(\nu_\e(y))}\frac{|\ln (\nu_\e(x))-\ln(\nu_\e(y))|^2}{|x-y|^{2+a}}dxdy\\
	&\quad + \int_{|x-y|\geq 1} \lt(\nu_\e(x)+\nu_\e(y)\rt)\frac{|\ln (\nu_\e(x))-\ln(\nu_\e(y))|}{|x-y|^{2+a}}dxdy\\
	&\quad \leq \int_{|x-y|\leq 1} \lt(\nu_\e(x)+\nu_\e(y)\rt)\frac{|\ln (\nu_\e(x))-\ln(\nu_\e(y))|^2}{|x-y|^{2+a}}dxdy\\
	&\quad + \int_{|x-y|\geq 1} \lt(\nu_\e(x)+\nu_\e(y)\rt)\frac{|\ln (\nu_\e(x))-\ln(\nu_\e(y))|}{|x-y|^{2+a}}dxdy,
	\end{split}
	\]
	Using \cite[Lemma 5.8]{HM} (see also Lemma \ref{lem:reg} below) we find that $\|\nabla \ln\nu_\e\|_{L^{\infty}(\R^2)} \leq \e^{-1}$. Therefore
	\[
	\II_a^N(\nu^{N,\e})\leq 2\int_{\R^2} \nu_\e(x)\lt(\int_{\R^2}(|x-y|^{-a}\wedge |x-y|^{-(1+a)})dy\rt)dx<\infty.
	\]
	
\end{proof}

\begin{remark}
	All the properties established on the fractional Fisher information can be proved with the same techniques for the classical Fisher information 
	\[
	\II^N(F^N) = \frac{1}{N}\sum_{i=1}^N\int_{ \R^{2N} }  \nabla_i F^N\cdot \nabla_i \ln F^NdX^N,
	\]	
	using this particular form, and a slight modification of point $(i)$ of Lemma \ref{lem:use1} which reads
	\[
	\nabla(fg)\cdot\nabla\ln g= \nabla f \cdot \nabla g+ \frac{f}{g}\lt|\nabla g \rt|^2\geq \nabla f \cdot \nabla g,
	\]	
	for $f,g$ two nonnegative functions. Classically those properties are obtained with the duality form 
	\[
	I(F^N)=\sup_{\varphi \in C_b^1(\R^{2N})} \left\langle F^N,   -\frac{\psi^2}{4}-\nabla\cdot \psi\right\rangle, 
	\]
	see \cite[Lemma 3.7]{HM} for instance. Also note that the statement and proof of Proposition \ref{prop:FFI} is also valid in dimension $d>2$.
\end{remark}

Accordingly to \cite[Definition 5.2]{HM}, define for $\kappa\geq 1$ the sets $\PP_\kappa(\PP(\R^2))$ as
\[
\begin{aligned}
\PP_\kappa(\PP(\R^2))&:=\{\pi \in \PP(\PP(\R^2))\,|\, \int_{\PP(\R^2)}\int_{\R^2}\lal x \ral^{\kappa}\rho(dx)\pi(d\rho)<\infty \}\\
&\quad := \{\pi \in \PP(\PP(\R^2))\,|\, \int_{\PP(\R^2)}M_\kappa(\rho)\pi(d\rho)<\infty \}\\
&\quad := \{\pi \in \PP(\PP(\R^2))\,|\, \MM_{\kappa}(\pi)<\infty \},
\end{aligned}
\]
and
\[
\PP_\kappa(\R^{2N}):=\{(F^N)_{N\geq 1}\,| F^N \in \PP_{\mbox{sym}}(\R^{2N})\, \mbox{and} \, \sup_{N\geq 1}\int_{\R^{2N}}\lal x_1 \ral^{\kappa}F^N(dx_1,\cdots,dx_N)<\infty \}.
\]
In some sense we abuse the $N$ in the notation, so we emphasize that $\PP_\kappa(\R^{2N})$ is a set of sequences. For $\pi \in \PP_{\kappa}(\PP(\R^2))$, define $(\pi_N)_{N\geq 1}$ its Hewitt and Savage (see for instance \cite[Theorem 5.1]{HM}) projection on $\PP_{\kappa}(\R^{2N})$ as
\[
\pi_N:=\int_{\PP(\R^2)}\rho^{\otimes N} \pi(d\rho).
\]
 Next define on $\PP_{\kappa}(\PP(\R^2))$ the functional $\tilde{\II}_a$ as 
	\bq
	\label{eq:FFIN}
	\tilde{\II}_a(\pi)=\sup_{N \geq 1}\II^N_a(\pi_N)=\lim_{N \geq 1}\II^N_a(\pi_N).
	\eq
	The fact that the $\lim$ equals the $\sup$ comes from the fact that the sequence $\lt(\II_a^N(\pi_N)\rt)_{N\geq 1}$ is nondecreasing. Indeed the sequence of symmetric probability measure $(\pi_N)_{\geq 1}$ is compatible, the marginal on $\R^{2(N-1)}$ of $\pi_N$ is $\pi_{N-1}$ and we use point $(iv)$ of Proposition \ref{prop:FFI} to conclude.\\	
	
	We now give the last technical result, which proof is delayed in appendix, and which enables to conclude to the desired $\Gamma$-l.s.c. property in the
	
	\begin{proposition}
		\label{lem:aff}
		The functional $\tilde{\II}_a$ defined in \eqref{eq:FFIN} is affine in the following sense. For any $\pi \in \PP_{\kappa}(\PP(\R^2))$ and any partition of $\PP_{\kappa}(\R^2)$ by some sets $(\omega_i)_{i=1,\cdots,M}$, such that $\omega_i$ is an open set in $\PP_{\kappa}(\R^2)\setminus\lt(\bigcap_{j=1}^{i-1}\omega_j\rt)$ for all $1\leq i\leq M-1$ and $\pi(\omega_i)>0$ for all $1\leq i \leq M$, defining 
		\begin{align*}
		&\alpha_i:=\pi(\omega_i) \quad \mbox{and} \quad \gamma_i=(\alpha_i)^{-1}\mb_{\omega_i}\pi \in \PP_\kappa(\PP(\R^2))\\
		&\gamma_i=0 \quad \text{if} \quad \alpha_i=0,
		\end{align*}
		so that 
		\[
		\pi=\sum_{i=1}^M \alpha_i\gamma_i, \quad \mbox{and} \quad \sum_{i=1}^M\alpha_i=1,
		\]
		it holds 
		\[
		\tilde{\II}_a(\pi)=\sum_{i=1}^M\alpha_i\tilde{\II}_a(\gamma_i).
		\]
	\end{proposition}	
	We now can state the key argument which enables in the Diffusion Dominated case to go beyond a convergence/consistency result, and provide a complete propagation of chaos result.
	\begin{corollary}
		\label{lem:gam}
		For any $\pi\in \PP_\kappa(\PP(\R^2))$ it holds
		\[
		\tilde{\II}_a(\pi)=\int_{\PP(\R^2)}I_a(\rho)\pi(d\rho),
		\]
		Moreover the functional $\tilde{\II}_a$ is affine, proper and l.s.c. w.r.t. the weak convergence in $\PP_{\kappa}(\PP(\R^2))$ and satisfies the $\Gamma$-lower semi continuous property, i.e. for any sequence $(F_N)_{N\geq 1}\in \PP_{\kappa}(\R^{2N})$ converging toward $\pi\in \PP_{\kappa}(\PP(\R^2))$ in the sense that
		\[
		\forall j\geq 2, F_N^j\overset{*}{\rightharpoonup}  \pi^j \,\mbox{in} \,\PP(\R^{2j}),
		\]
		where $F_N^j$ denotes the marginal on $\R^{2j}$ of $F_N$, and 
		$\pi^j$ the $\pi$ Hewitt and Savage projection on $\R^{2j}$, then it holds
		\[
		\tilde{\II}_a(\pi)=\int_{\PP(\R^2)}I_a(\rho)\pi(d\rho)\leq \liminf_{N}\II_a^N(F_N).
		\]
		
		\begin{proof}
			This result is an immediate consequence of \cite[Lemma 5.6]{HM}, and Propositions \ref{prop:FFI} and \ref{lem:aff}. We leave the reader check that the last two propositions consist in checking that the family of functionals $(\II_a^N)_{N\geq 1}$ satisfies the assumptions of \cite[Lemma 5.6]{HM}.
		\end{proof}	
		
	\end{corollary}
		This result also provides the so-called $3$-level representation of the fractional Fisher information in $2$d defined in \eqref{eq:FFIN} i.e.
		\[
		\tilde{\II}_a(\pi)=\int_{\PP(\R^2)}I_a(\rho) \pi(d\rho)=\sup_{N\geq 1}\II^N_a(\pi^N)=\lim_{N\geq 1}\II^N_a(\pi^N),
		\] 
		with $\pi^N=\int_{\PP(\R^2)}\rho^{\otimes N} \pi(d\rho)$. 
		Note that this result in the classical case, was one of the novelty provided in \cite{HM}.		
		\subsection{Convergence/consistency of particle system \eqref{eq:part_fr_KS}}
In this section we establish the tightness of the law of the particle system \eqref{eq:part_fr_KS} in both case $\alpha< a$ and $\alpha=a$. Note that in both case we have that
\[
X_t^1=X_0^1+\int_0^t \frac{1}{N}\sum_{j>1}K_{\alpha}(X_s^1-X_s^j)\,ds+ \ZZ^1_t :=X_0^1+J_t^{N,1}+ \ZZ^1_t, 
\]
so that it is enough to show the tightness of the $\lt((J_t^{N,1})_{t\in [0,T]}\rt)_{N\geq 1}$ to deduce the tightness of the law of $\lt((\mei \delta_{X_t^i})_{t\in [0,T]}\rt)_{N\geq 1}$, due to point $(ii)$ of Proposition \ref{prop:chaos}. First we need some moments estimates given in the 

\begin{lemma}
	\label{lem:mom}
	Let $(X_t^1,\cdots,X_t^N)$ be a solution to \eqref{eq:part_fr_KS} with $\alpha\in (1,a]$ (with law $F_t^N$). Then for any $\kappa\in (1,a)$ and $t>0$ there is a constant $C_{a,\kappa,t}$ such that
	\[
	\sup_{i= 1,\cdots,N}	\E\lt[ \lal X_t^i\ral^{\kappa}\rt]=\int_{\R^{2N}} \lal x_1 \ral^\kappa F_t^N\leq C_{a,\kappa}t+\int_{\R^{2N}} \lal x_1 \ral^\kappa F_0^N,
	\]
\end{lemma}

\begin{proof}
	Using Ito's formula with the $\CC^2$ functional $m\kappa$ yields
	\[
	\begin{split}
	m_\kappa(X_t^1)& =m_\kappa(X_0^1)-\int_0^t \frac{\chi}{N} a \lt(1+ \lt|X_s^1\rt|^2\rt)^{\frac{\kappa-2}{2} }X_s^{1}\cdot \lt(\sum_{k>1} \frac{Z_s^{1,k}}{|Z_s^{1,k}|^\alpha}\rt)\,ds\\
	&\quad +\int_{[0,t]\times \R^2} \lt(m_{\kappa}(X_{s_-}^{1}+x)-m_{\kappa}(X_{s_-}^{1})-x\cdot \nabla m_{\kappa}(X_{s_-}^{1})  \rt) M_1(ds,dx)\\
	&\quad +\int_{[0,t]\times \R^2} x\cdot \nabla m_{\kappa}(X_{s_-}^{1})   \bar{M}_1(ds,dx).
	\end{split}
	\]
	Taking the expectation yields
	\[
	\begin{split}
	\E\lt[m_{\kappa}(X_{t}^{1})\rt]& =\E\lt[m_{\kappa}(X_{0}^{1})\rt]-\int_0^t \frac{\chi}{N} a \E\lt[\lt(1+ \lt|X_s^1\rt|^2\rt)^{\frac{\kappa-2}{2}} X_s^{1}\cdot \lt(\sum_{k> 1} \frac{Z_s^{1,k}}{|Z_s^{1,k}|^\alpha}\rt)\rt]\,ds\\
	&\quad +\int_0^t\E\lt[c_{2,a}\int_{\R^2}\frac{m_{\kappa}(X_{s}^{1}+x)-m_{\kappa}(X_{s}^{1})-x\cdot \nabla m_{\kappa}(X_{s}^{1}) }{|x|^{2+a}}dx\rt]\,ds.
	\end{split}
	\]
	By ex-changeability and point $(iii)$ of Lemma \ref{lem:use1} we get
	\[
	\E\lt[\lt(1+ \lt|X_s^1\rt|^2\rt)^{\frac{\kappa-2}{2} }X_s^{1}\cdot  \frac{Z_{s}^{1,k}}{|Z_{s}^{1,k}|^\alpha}\rt]=\frac{1}{2}\E\lt[\lt(\lt(1+ \lt|X_s^1\rt|^2\rt)^{\frac{\kappa-2}{2} }X_s^{1}-\lt(1+ \lt|X_s^k\rt|^2\rt)^{\frac{\kappa-2}{2} }X_s^{k}\rt)\cdot  \frac{Z_{s}^{1,k}}{|Z_{s}^{1,k}|^\alpha}\rt]\geq 0,
	\]
	and then
	\[
	\E\lt[m_{\kappa}(X_{t}^{1})\rt]\leq \E\lt[m_{\kappa}(X_{0}^{1})\rt]+\int_0^t\E\lt[-(-\Delta)^{a/2}m_{\kappa}(X_{s}^{1})\rt]\,ds.
	\]
	Using \eqref{eq:momfralap} and ex-changeability of the particles yields the desired conclusion.  
\end{proof}

Then let be $0<s<t<T$ and note that for any $p\in\lt(1,\frac{a}{\alpha-1}\rt)$
\bq
\label{eq:cont}
\begin{split}
\lt|J_t^{N,1}-J_s^{N,1} \rt|&\leq \frac{1}{N}\sum_{j>1}\int_s^t \lt|X_u^1-X_u^j   \rt|^{1-\alpha}\, du\\
&\quad \leq |t-s|^{(p-1)/p} \frac{1}{N}\sum_{j>1} \lt(\int_0^T|X_u^{1}-X_u^j|^{(1-\alpha)p}\rt)^{1/p}\\
&\quad \leq |t-s|^{(p-1)/p} \lt(1+\frac{1}{N}\sum_{j>1} \int_0^T|X_u^{1}-X_u^j|^{(1-\alpha)p}\,du\rt):=|t-s|^{\beta}Z_{N,p}^T.
\end{split}
\eq
Therefore, following a standard procedure (see for instance \cite{GQ,FJ,MSN}), it is enough to show that 
\[
\sup_{N\in \mathbb{N}}\E\lt[Z_{N,p}^T\rt]<\infty,
\]
to deduce the tightness of the $(J_t^{N,1})_{t\in [0,T]}$. Also note that by ex changeability we get
\bq
\label{eq:ExpTi}
\E\lt[Z_{N,p}^T\rt]\leq 1+ \int_0^T \E\lt[ |X_u^{1}-X_u^2|^{-(\alpha-1)p} \rt]\, du.
\eq
We now give the two main results of this section. The first is the 
\begin{proposition}
	\label{prop:tiDD}
	Let be $1<\alpha<a<2$, and $(X_t^1,\cdots,X_t^N)$ a solution to equation \eqref{eq:part_fr_KS} for an initial condition with law $(F^N_0)_{N\geq 1}\in \mathcal{P}_{\kappa}(\R^{2N})$ for some $\kappa\in \lt(1, a\rt)$ and such that 
	\[
	\HH^N(F_0^N)=\frac{1}{N}\int_{\R^{2N}} F_0^N\ln F_0^NdX^N<\infty.
	\]
	Then for any $t>0$ it holds
	\[
	\int_0^t \II_a^N(F_s^N)ds\leq C\lt(\HH^N(F_0^N)+\int_{\R^{2N}}\lal x_1\ral^\kappa F_0^N+t\rt),
	\]
	where $C>0$ is a constant independent of $F_0^N$. In particular for any $\gamma\in(0,a)$ it holds
	\[
	\sup_{N\geq 1}\int_0^T\sup_{1\leq i\neq j\leq N}\E\lt[|X_u^i-X_u^j|^{-\gamma }\rt]\,du <\infty.
	\]
	
\end{proposition}

The proof of this proposition is based on \cite{GQ} itself inspired by \cite{MSN}. It relies on a control of the Fisher information. The second result proved in this section is given in the 
\begin{proposition}
	\label{prop:tiFC}
	Let be $1<\alpha=a<2$ and $(X_t^1,\cdots,X_t^N)$ a solution to equation \eqref{eq:part_fr_KS} for an initial condition with law $(F^N_0)_{N\geq 1}\in \mathcal{P}_{\kappa}(\R^{2N})$ for some $\kappa\in \lt(1, a\rt)$.
	There exists $a^*\in(1,2)$, and $\chi_a:(a^*,2)\mapsto (0,\infty)$ with $\lim_{a\rightarrow 2^-}\chi_a=1$, such that if $a\in (a^*,2)$ and $\chi\in (0,\chi_a)$ then it holds for any $T>0$ and $\e\in(0,1)$
	\[
	\sup_{N\geq 1}\int_0^T\sup_{1\leq i\neq j \leq N}\E\lt[|X_u^i-X_u^j|^{\e-a }\rt]\,du <\infty.
	\]
\end{proposition}
The proof of this proposition is based on \cite{FJ} itself inspired by \cite{Osa}. Unfortunately, the estimates given here are not sharp, and it seems not clear how to improve them. The proofs of both these propositions are given later in this section. 
	\begin{remark}
	Note that we have not said anything so far regarding the existence of solutions to equation \eqref{eq:part_fr_KS}. The difficulty comes from the non smoothness of the drift coefficient. However this can be solved by mollifying the interaction kernel, and showing that the family of the (unique due to classical SDE theory) solution with such a mollified drift is tight in the mollification parameter for fixed $N$. But as the computations done in the proofs of Proposition \ref{prop:tiDD} and \ref{prop:tiFC} show the tightness uniformly in $N$ with the not mollified interaction kernel, they a fortiori show the tightness in the regularization parameter for the regularized system with fixed number of particles. Hence it is a standard procedure (see the \cite[Theorem 5]{FJ}) to build a solution to the particle system thanks to the tightness argument. We leave the reader check that the less singular kernel or the $a$-stable L\'evy noise considered here instead of the Newtonian force and Brownian motion considered in \cite{FJ}, do not change the argument used by Fournier and Jourdain. However, this argument does not provide uniqueness, but it is not required in order to obtain Theorem \ref{thm:main}.
\end{remark}
\subsubsection{Proof of Proposition \ref{prop:tiDD}}
We begin the proof of this proposition with some fractional logarithmic Gagliardano-Nirenberg-Sobolev inequality. More precisely we have the

\begin{lemma}
	\label{lem:GNS}
	Let be $d\geq 2$, for any $p\in (1,\frac{d}{d-a}]$, there is a constant $C_{p,d,a}>0$ s.t. $\forall u\in \mathcal{P}(\R^d) $ it holds
	  $$\|u\|_{L^{p}(\R^d)}\leq C_{p,d,a}I_a(u)^{1-\frac{d}{a}\lt(\frac{1}{p}-\frac{d-a}{d}\rt)}.$$ 
\end{lemma}	

\begin{remark}
	This Lemma can be seen as a generalisation in the fractional case of \cite[Lemma 3.2]{MSN}, in the case $\alpha<2$. However in the case $d=2$, the critical exponent $\frac{d}{d-a}$ can be reached contrary to the case $a=2$.
\end{remark}

\begin{proof}
	
	First recall the classical inequality for $x,y\geq 0$
	\[
	\Phi(x,y)\geq 4 \lt( \sqrt{x}-\sqrt{y} \rt)^2, 
	\]
	so that
	\begin{align*}
	I_a(u)=&\int_{\R^d\times\R^d}\frac{\Phi(u(x),u(y))}{|x-y|^{d+a}}\\
	&\geq 4\int_{\R^d\times\R^d}\frac{\lt( \sqrt{u(x)}-\sqrt{u(y)} \rt)^2}{|x-y|^{d+a}}=4|\sqrt{u}|^2_{H^{a/2}(\R^d)}.
	\end{align*}
	By fractional Sobolev's embeddings (see for instance \cite[Theorem 6.5]{DiNez}) there is $C_{a,d}>0$ such that
	\[
	|\sqrt{u}|^2_{H^{a/2}(\R^d)}\geq C_{a,d} \|\sqrt{u}\|^2_{L^{\frac{2d}{d-a}}(\R^d)}= C_{a,d}\|u\|_{L^{\frac{d}{d-a}}(\R^d)},
	\]
	we conclude to the desired result since for any $p\in (1,\frac{d}{d-a}]$ by interpolation inequality
	\[
	\|u\|_{L^p(\R^d)}\leq \|u\|^{\frac{d}{a}\lt(\frac{1}{p}-\frac{d-a}{d}\rt)}_{L^1(\R^d)}\|u\|_{L^{\frac{d}{d-a}}(\R^d)}^{1-\frac{d}{a}\lt(\frac{1}{p}-\frac{d-a}{d}\rt)}.
	\]
\end{proof}

	\begin{lemma}
		\label{lem:FiExp}
		Let be $\gamma\in (0,a)$ and $p\in \lt(\frac{2}{2-\gamma},\frac{2}{2-a}\rt)$. There exists a constant $C_{\gamma,p,a}>0$ such that for any $F^2\in\mathcal{P}_{\mbox{sym}}(\R^{2d})$ it holds	
		\[
		\int_{\R^d\times \R^d} |x_1-x_2|^{-\gamma}F^2(dx_1,dx_2)\leq C_{\gamma,p,a}\lt(1+\II^2_a(F^2)^{1-\frac{2}{a} \lt(\frac{1}{p}-\frac{2-a}{2}\rt) }\rt).
   		\]

	\end{lemma}

	\begin{proof}
		We introduce the unitary linear transformation 
		\[
		\forall(x_1,x_2)\in \R^2\times \R^2,\quad \Psi(x_1,x_2)=(x_1-x_2,x_2).
		\]
		Denote $F^2=\LL(X_1,X_2)$ and $\tilde{F}^2=F^2\circ \Psi^{-1}$, which is nothing but the law of $\Psi(X_1,X_2)$. A simple substitution shows that $\II_a^2(F^2)=\II_a^2(\tilde{F}^2)$. Indeed
		\[
		\begin{aligned}
		\II_a^2(\tilde{F}^2)=&\int_{\R^2}\int_{\R^2\times \R^2}\frac{\Phi(F^2\circ\Psi^{-1}(x,z),F^2\circ\Psi^{-1}(y,z))}{|x-y|^{2+a}}dxdydz\\
		&\quad =\int_{\R^2}\int_{\R^2\times \R^2}\frac{\Phi(F^2(x+z,z),F^2(y+z,z))}{|(x+z)-(y+z)|^{2+a}}dxdydz\\
		&\quad =\int_{\R^2}\int_{\R^2\times \R^2}\frac{\Phi(F^2(x,z),F^2(y,z))}{|x-y|^{2+a}}dxdydz=\II_a^2(F^2).		
		\end{aligned}
		\]
		Then

	 \[
	 \begin{split}
	  \int_{\R^{2}\times \R^2} |x_1-x_2|^{-\gamma} F^2(x_1,x_2)dx_1dx_2
	 &\quad = \int_{\R^{2}\times \R^2} |y_1|^{-\gamma}\tilde{F}^2(y_1,y_2)dy_1dy_2\\	 
	&\quad \leq 1+\int_{\R^{2}\times \R^2} \mb_{|y_1|\leq 1}|y_1|^{-\gamma}\tilde{F}^2(y_1,y_2)dy_1dy_2\\
	&\quad \leq  1+\int_{|y_1|\leq 1}|y_1|^{-\gamma}\tilde{f}_1^2(y_1)dy_1\\
	\end{split}
	\]		
	where we did the change of variable $(y_1,y_2)=\Psi(x_1,x_2)$, $\tilde{f}_1^2$ denotes the first marginal of $\tilde{F}^2$. Then for any $p>\frac{2}{2-\gamma}$

	\[
	\begin{split}
	\int_{\R^{2}\times \R^2} |x_1-x_2|^{-\gamma} F^2(x_1,x_2)dx_1dx_2&\leq 1+\lt(\int_{|y_1|\leq 1}|y_1|^{-\gamma p'}dy_1 \rt)^{1/p'}\|\tilde{f}_1^2\|_{L^p(\R^2)}\\
	&\quad \leq C_{\gamma,\beta}\lt(1+I^{1-\frac{2}{a} \lt(\frac{1}{p}-\frac{2-a}{2}\rt) }_a(\tilde{f}_1^2)  \rt),
	\end{split}
	\]	
	which concludes the proof since
	\[
	I_a(\tilde{f}_1^2)\leq \II^2_a(\tilde{F}^2)= \II^2_a(F^2)\leq \II^N_a(F^N),
	\]
	thanks to point $(iv)$ Proposition \ref{prop:FFI}.	
	\end{proof}	

\begin{remark}
	Note that in the tensorized case $F^2=f_1\otimes f_2$ (in which we are clearly not), Hardy-Littlewood-Sobolev's inequality yields
	\begin{align*}
	\int_{\R^d\times \R^d} |x_1-x_2|^{-\gamma}F^2(dx_1,dx_2)&\leq C \|f_1\|_{L^{p}}\|f_2\|_{L^{p}}, \ \text{with} \  \frac{2}{p}=2-\frac{\gamma}{d},
	\end{align*}
	and by the same argument as above
	\begin{align*}
	\int_{\R^d\times \R^d} |x_1-x_2|^{-\gamma}F^2(dx_1,dx_2)\leq C \lt(\II^2_a(F^2)\rt)^{\gamma/a},
	\end{align*}
	which holds even in the critical case $\gamma=a$, provided that $a<d$. The latter condition exclude the classical case $a=d=2$.
\end{remark}

\begin{corollary}
	\label{cor:FiExp}
	Let $(X_1,\cdots,X_N)\in\R^{2N}$ be a random vector of law $F^N\in\mathcal{P}_{\mbox{sym}}(\R^{2N})$. Let be $\gamma\in (0,a)$ and $p\in \lt(\frac{2}{2-\gamma},\frac{2}{2-a}\rt)$. There exists a constant $C_{\gamma,p,a}>0$ such that 
	\[
	\sup_{1\leq i\neq j\leq N} \E\lt[ |X_i-X_j|^{-\gamma} \rt]\leq C_{\gamma,p,a}\lt(1+\II^N_a(F^N)^{1-\frac{2}{a} \lt(\frac{1}{p}-\frac{2-a}{2}\rt) }\rt).
	\]

\end{corollary}
This is a simple consequence of exchangeability and point $(iv)$ of Proposition \ref{prop:FFI}. We now have all the ingredient to mimic the entropy dissipation estimate of \cite{MSN}. Precisely we have the

\begin{lemma}
	\label{lem:EntDis}
	Let be $\alpha \in (1,a)$, $\chi>0$, and $(X_1^t,\cdots,X_t^N)$ a solution to \eqref{eq:part_fr_KS} with initial law $F_0^N\in \PP_{\kappa}(\R^{2N})$ for some $\kappa\in (1,a)$, and denote $F_t^N\in \mathcal{P}_{\kappa}(\R^{2N})$ its law. Then for any $p\in\lt(\frac{2}{2-\alpha},\frac{2}{2-a}\rt)$ there is a constant $C_{\alpha,p,a,\chi,\kappa}>0$ such that
	\[
	\int_0^t \II_a^N(F_s^N)\,ds\leq 	C_{\alpha,p,a,\chi,\kappa}\lt(\HH^N(F_0^N)+\int_{\R^{2N}}\lal x_1\ral^\kappa F_0^N+ t\rt).
	\]
\end{lemma}

\begin{proof}
Let be $\phi\in \mathcal{C}^{\infty}(\R^{2N})$. Due to Ito's rule we get
\[
\begin{split}
\phi(X_t^1,\cdots,X_t^N)&=\phi(X_0^1,\cdots,X_0^N)+\int_0^t \sum_{i=1}^N \nabla_{i}\phi(X_s^1,\cdots,X_s^N)\cdot 
\frac{1}{N}\sum_{j\neq i} K_\alpha(X_s^i-X_s^j)\,ds\\
&\quad + \sum_{i=1}^N\int_{[0,t]\times \R^2} \lt(\phi(X_{s_-}^1,\cdot,X_{s_-}^{i}+z,\cdot,X_{s_-}^N)-\phi(X_{s_-}^1,\cdots,X_{s_-}^N)-z\cdot\nabla_i\phi(X_{s_-}^1,\cdots,X_{s_-}^N)\rt)M_i(ds,dz)\\
&\quad + \sum_{i=1}^N\int_{[0,t]\times \R^2} z\cdot\nabla_i\phi(X_{s_-}^1,\cdots,X_{s_-}^N)\bar{M}_i(ds,dz),
\end{split}
\]
denoting $F_t^N=\LL(X_t^1,\cdots,X_t^N) $ taking the expectation yields
and using some integration by parts, we find that $F_.^N$ solves in the weak sense the following linear PDE
\[
\begin{split}
&\pa_t F_t^N+ \sum_{i=1}^N \nabla_i \cdot \lt( 
\frac{1}{N} \sum_{j\neq i} K_\alpha(x_i-x_j) F_t^N  \rt)\\
&\quad+\sum_{i=1}^N c_{2,a}\mbox{p.v.}\int_{\R^2} \frac{F_t^N(x_1,\cdot,x_i+z,\cdot,x_N)-F_t^N(x_1,\cdots,x_N)}{|z|^{2+a}}dz=0.
\end{split}
\]
Hence we deduce, dropping the $t$ in the notation for the sake of simplicity
\[
\begin{split}
\frac{d}{dt}\HH^N(F^N)&=\frac{1}{N}\int_{ \R^{2N} }\pa_tF^N (1+\ln F^N ) \\
&\quad = -\frac{\chi}{N^2}\sum_{j\neq i}\int_{ \R^{2N} } \nabla_i \cdot \lt( 
K_\alpha(x_i-x_j) F^N \rt) (1+\ln F^N )\\
&\quad - \frac{1}{N}\sum_{i=1}^N \int_{ \R^{2N} }\mbox{p.v.}\int_{\R^2} \frac{F^N(x_1,\cdot,x_i+z,\cdot,x_N)-F^N(x_1,\cdots,x_N)}{|z|^{2+a}}(1+\ln F^N(x_1,\cdots,x_N) )dzdx_1 \cdots dx_N\\
&\quad = \frac{\chi}{N^2}\sum_{j\neq i}\int_{ \R^{2N} } \nabla_i \cdot \lt( 
K_\alpha(x_i-x_j)  \rt)F^N(x_1,\cdots,x_N)dx_1,\cdots,dx_N\\
&-\frac{1}{N}\sum_{i=1}^N \int_{ \R^{2N} }
\int_{\R^2} \frac{\lt(F^N(x_1,\cdot,x,\cdot,x_N)-F^N(x_1,\cdot,y,\cdot,x_N)\rt)\lt(\ln F^N(x_1,\cdot,x,\cdot,x_N)-\ln F^N(x_1,\cdot,y,\cdot,x_N)\rt)}{|x-y|^{2+a}}dxdyd\tilde{X}^i_N\\
&\quad= \frac{\chi(N-1)(\alpha-1)}{N}\int_{\R^{2}\times\R^{2}} |x_1-x_2|^{-\alpha}F^2(x_1,x_2)dx_1dx_2 - \II^N_a(F^N),
\end{split}
\]
where $\tilde{X}^i_N=dx_1,\cdot,dx_{i-1},dx_{i+1},\cdot,dx_N$ and $F^2\in \mathcal{P}(\R^2\times\R^2)$ stands for the two particles marginal of $F^N$. But using Corollary \ref{cor:FiExp}, we find for any $p\in \lt(\frac{2}{2-\alpha},\frac{2}{2-a}\rt)$ 
\[
\begin{split}
\frac{d}{dt}\HH^N(F_t^N)&\leq \chi(\alpha-1)C_{\alpha,p,a}\lt(1+\II^N_a(F_t^N)^{1-\frac{2}{a} \lt(\frac{1}{p}-\frac{2-a}{2}\rt) }\rt)- \II^N_a(F^N_t)\\
&\quad \leq \chi(\alpha-1)C_{\alpha,p,a} -\lt(\II^N_a(F^N_t)-\chi(\alpha-1)C_{\alpha,p,a}\II^N_a(F^N_t)^{1-\frac{2}{a} \lt(\frac{1}{p}-\frac{2-a}{2}\rt) }\rt)\\
&\quad  \leq C_{\alpha,p,a,\chi}-\frac{1}{2}\II^N_a(F^N_t),
\end{split}
\]
so that
\bq
\label{eq:dem1}
\HH^N(F_t^N)+\int_0^t \frac{1}{2}\II^N_a(F^N_s)ds\leq \HH^N(F_0^N)+ C_{\alpha,p,a,\chi}t.
\eq
Then define $G^N_{\kappa,\lambda}=e^{-\sum_{i=1}^N \lambda_\kappa|x_i|^\kappa}$, with $\lambda_\kappa>0$ being such that
\[
\int_{\R^2} e^{-\lambda_\kappa|x|^{\kappa}} dx=1,
\]
then with $h(s)=s\ln s -s+ 1\geq 0$
\begin{align*}
\int_{\R^{2N}}F_t^N\ln F_t^N&=\int_{\R^{2N}}G^N_{\kappa,\lambda} h\lt( \frac{F_t^N}{G^N_{\kappa,\lambda}a} \rt)+\int_{\R^{2N}}F_t^N\ln G^N_{\kappa,\lambda}\\
&\geq -\lambda_\kappa\sum_{i=1}^N\int_{\R^{2N}}|x_i|^\kappa F_t^N.
\end{align*}
So that by symmetry
\[
\HH^N(F_t^N)\geq-\lambda_\kappa\int_{\R^{2N}}\lal x_1\ral^\kappa F_t^N
\]
and summing $\lambda_\kappa\int_{\R^{2N}}\lal x_1\ral^\kappa F_t^N$ to \eqref{eq:dem1}, combined to Lemma \ref{lem:mom} yields 
\begin{align*}
\lt(\HH^N(F_t^N)+\lambda_\kappa\int_{\R^{2N}}\lal x_1\ral^\kappa F_t^N \rt)+\int_0^t \frac{1}{2}\II^N_a(F^N_s)ds\leq \HH^N(F_0^N)+\lambda_\kappa\int_{\R^{2N}}\lal x_1\ral^\kappa F_0^N+\lt(C_{\alpha,p,a,\chi}+\lambda_\kappa C_{a,\kappa}\rt)t
\end{align*}
which concludes the desired result, since the l.h.s. of the above inequality is the sum of two nonnegative term. 
\end{proof}

Combining Lemmas \ref{lem:EntDis} and \ref{lem:GNS} concludes the proof of Proposition \ref{prop:tiDD}.

\subsubsection{Proof of Proposition \ref{prop:tiFC}}

In this section now set $\alpha=a$. In this case we extend the method used in \cite{FJ}. In this case let $(X^1_t\cdots,X_N^t)_{t\in[0,T]}$ be a solution to \eqref{eq:part_fr_KS} and denote $Z^{i,j}_s:=X_s^i-X_s^j$ note that it solves 
\[
Z_t^{i,j}=Z_0^{i,j}-\frac{\chi}{N}\int_0^t \sum_{k\neq i,j}\lt( \frac{Z_s^{i,k}}{|Z_s^{i,k}|^a}-\frac{Z_s^{j,k}}{|Z_s^{j,k}|^a}\rt)\, ds -\frac{2\chi}{N}\int_0^t \frac{Z_s^{i,j}}{|Z_s^{i,j}|^a}\, ds+\int_{[0,t]\times \R^2}x \lt(\bar{M}_i-\bar{M}_j\rt)(ds,dx).
\]
Denote 
\[
H^{i,j}_t:= \int_{[0,t]\times \R^2}x \lt(\bar{M}_i-\bar{M}_j\rt)(ds,dx).
\]
It holds (see \cite{Cin}) for any $r\in \R^2$, since $M_i$ is independent of $M_j$
\[
\begin{split}
\E\lt[e^{i r\cdot H^{i,j}_t} \rt]&=\E\lt[e^{i r\cdot\int_{[0,t]\times \R^2}x \bar{M}_i(ds,dx)} \rt]\E\lt[e^{-i r\cdot\int_{[0,t]\times \R^2}x \bar{M}_j(ds,dx)} \rt]\\
&\quad = e^{-t 2 c_{a,2}|r|^a},
\end{split}
\]
hence it deduces that
\[ 
H^{i,j}_t\overset{(\mathcal{L})}{=}\int_{[0,t]\times \R^2}2^{1/a}x\bar{M}_i(ds,dx).
\]
Let be $\e\in (0,1)$, $\eta>0$ and similarly as in Lemma \ref{lem:frlap} define 
\[
\phi_{\eta}(x)=\eta^\e \lal \frac{x}{\eta} \ral^\e= \lt(|x|^2+\eta^2\rt)^{\frac{\e}{2}},
\] 
using Ito's rule yields
\bq
\label{eq:ITO}
\begin{split}
\phi_{\eta}\lt(Z_t^{i,j}\rt)& =\phi_{\eta}\lt(Z_0^{i,j}\rt)-\int_0^t \frac{\chi}{N} \e \lt(|Z_s^{i,j}|^2+\eta^2\rt)^{\frac{\e-2}{2}}Z_s^{i,j}\cdot \lt(\sum_{k\neq i,j}\lt( \frac{Z_s^{i,k}}{|Z_s^{i,k}|^a}-\frac{Z_s^{j,k}}{|Z_s^{j,k}|^a}\rt)\rt)\, ds\\
&\quad- \int_0^t \frac{2\chi\e}{N} \lt(|Z_s^{i,j}|^2+\eta^2\rt)^{\frac{\e-2}{2}}\lt|Z_s^{i,j}\rt|^{2-a}\,ds\\
&\quad +\int_{[0,t]\times \R^2} \lt(\phi_{\eta}\lt(Z_{s_-}^{i,j}+2^{1/a}x\rt)-\phi_{\eta}\lt(Z_{s_-}^{i,j}\rt)-2^{1/a}x\cdot\nabla\phi_{\eta}\lt(Z_{s_-}^{i,j}\rt)  \rt) M_i(ds,dx)\\
&\quad +\int_{[0,t]\times \R^2} 2^{1/a}x\cdot\nabla\phi_{\eta}\lt(Z_{s_-}^{i,j}\rt)   \bar{M}_i(ds,dx).
\end{split}
\eq
Taking the expectation yields 
\[
\begin{split}
\E\lt[\phi_{\eta}\lt(Z_t^{i,j}\rt)\rt]&=\E\lt[\phi_{\eta}\lt(Z_0^{i,j}\rt)\rt]-\int_0^t \frac{\chi}{N} \e \E\lt[\lt(|Z_s^{i,j}|^2+\eta^2\rt)^{\frac{\e-2}{2}}Z_s^{i,j}\cdot \lt(\sum_{k\neq i,j}\lt( \frac{Z_s^{i,k}}{|Z_s^{i,k}|^a}-\frac{Z_s^{j,k}}{|Z_s^{j,k}|^a}\rt)\rt)\rt]\, ds\\
&\quad- \int_0^t \frac{2\chi\e}{N} \E\lt[\lt(|Z_s^{i,j}|^2+\eta^2\rt)^{\frac{\e-2}{2}}|Z_s^{i,j}|^{2-a}\rt]\,ds\\
&\quad +\int_0^t\E\lt[c_{2,a}\int_{\R^2}\frac{\phi_{\eta}\lt(Z_{s}^{i,j}+2^{1/a}x\rt)-\phi_{\eta}\lt(Z_{s}^{i,j}\rt)-2^{1/a}x\cdot\nabla\phi_{\eta}\lt(Z_{s}^{i,j}\rt) }{|x|^{2+a}}dx\rt]\,ds.
\end{split}
\]
Note that the change of variable $x'=2^{1/a}x$ yields
\begin{align*}
c_{2,a}\int_{\R^2}\frac{\phi_{\eta}\lt(Z+2^{1/a}x\rt)-\phi_{\eta}\lt(Z\rt)-2^{1/a}x\cdot\nabla\phi_{\eta}\lt(Z\rt) }{|x|^{2+a}}dx&=c_{2,a}\int_{\R^2}\frac{\phi_{\eta}\lt(Z+x'\rt)-\phi_{\eta}\lt(Z\rt)-x'\cdot\nabla\phi_{\eta}\lt(Z\rt) }{|2^{-1/a}x'|^{2+a}}2^{-2/a}dx'\\
&=-2(-\Delta)^{a/2} \phi_\eta(Z)
\end{align*}
Also note that due to ex changeability we find
\[
\begin{split}
&A_s^{i,j}:=\E\lt[\lt(|Z_s^{i,j}|^2+\eta^2\rt)^{\frac{\e-2}{2}}Z_s^{i,j}\cdot \lt(\sum_{k\neq i,j}\lt( \frac{Z_s^{i,k}}{|Z_s^{i,k}|^a}-\frac{Z_s^{j,k}}{|Z_s^{j,k}|^a}\rt)\rt)\rt]\\
&\quad = \sum_{k\neq i,j} \lt( \E\lt[\lt(|Z_s^{i,j}|^2+\eta^2\rt)^{\frac{\e-2}{2}}Z_s^{i,j}\cdot  \frac{Z_s^{i,k}}{|Z_s^{i,k}|^a}\rt]-\E\lt[\lt(|Z_s^{i,j}|^2+\eta^2\rt)^{\frac{\e-2}{2}}Z_s^{i,j}\cdot  \frac{Z_s^{j,k}}{|Z_s^{j,k}|^a}\rt]\rt)\\
&\quad \leq \sum_{k\neq i,j} \lt( \E\lt[\lt(|Z_s^{i,j}|^2+\eta^2\rt)^{\frac{\e-2}{2}}\lt|Z_s^{i,j}\rt|\lt|Z_s^{i,k}\rt|^{1-a}\rt]+\E\lt[\lt(|Z_s^{i,j}|^2+\eta^2\rt)^{\frac{\e-2}{2}}\lt|Z_s^{i,j}\rt|\lt|Z_s^{j,k}\rt|^{1-a}\rt]\rt)\\
&\quad = 2\sum_{k\neq i}\E\lt[\lt(|Z_s^{i,j}|^2+\eta^2\rt)^{\frac{\e-2}{2}}\lt|Z_s^{i,j}\rt|\lt|Z_s^{i,k}\rt|^{1-a}\rt]\\
&\quad \leq 2\sum_{k\neq i} \E\lt[\lt(\lt(|Z_s^{i,j}|^2+\eta^2\rt)^{\frac{\e-2}{2}}\lt|Z_s^{i,j}\rt|\rt)^{p}\rt]^{1/p}\E\lt[\lt|Z_s^{i,k}\rt|^{(1-a)p/(p-1)}\rt]^{(p-1)/p}\\
&\quad \leq 2(N-2)\E\lt[\lt(\lt(|Z_s^{1,2}|^2+\eta^2\rt)^{\frac{\e-1}{2}}\rt)^{p}\rt]^{1/p}\E\lt[\lt|Z_s^{1,2}\rt|^{(1-a)p/(p-1)}\rt]^{(p-1)/p},
\end{split}
\] 
for any $p>1$. Choosing $p= 1+\frac{a-1}{1-\e}$ yields 
\[
A_s^{i,j}\leq 2(N-2)\E\lt[\lt(\lt(|Z_s^{1,2}|^2+\eta^2\rt)^{\frac{\e-1}{2}}\rt)^{\frac{a-\e}{1-\e}}\rt]^{\frac{1-\e}{a-\e}}\E\lt[\lt|Z_s^{1,2}\rt|^{\e-a}\rt]^{\frac{a-1}{a-\e}   }.
\]
Putting all those estimates together and using also Lemma \ref{lem:frlap}, we get
\bq
\label{eq:regcomp}
\begin{split}
\E\lt[\phi_{\eta}\lt(Z_t^{1,2}\rt)\rt]&\geq-\frac{2\chi\e(N-2) }{N}\int_0^t\E\lt[\lt(|Z_s^{1,2}|^2+\eta^2\rt)^{\frac{\e-a}{2}}\rt]^{\frac{1-\e}{a-\e}}\E\lt[\lt|Z_s^{1,2}\rt|^{\e-a}\rt]^{\frac{a-1}{a-\e}   }\,ds- \int_0^t \frac{2\chi\e}{N} \E\lt[\lt(|Z_s^{i,j}|^2+\eta^2\rt)^{\frac{\e-a}{2}}\rt]\,ds\\
&\quad +2c_{2,a} \e \int_0^t \lt(\frac{\e\pi}{2(2-a)} \E\lt[ \lt(|Z_s^{1,2}|^2+\eta^2\rt)^{\frac{\e-4}{2}}|Z_s^{1,2}|^{4-a} \rt] -2\pi C_{\e,a}\E\lt[|Z_s^{1,2}|^{\e-a} \rt] \rt) \,ds.
\end{split}
\eq
On the other hand, provided that $\eta\in(0,1)$, since $\kappa>1>\e$ using ex-changeability 
\[
\begin{split}
\E\lt[\phi_{\eta}\lt(Z_t^{1,2}\rt)\rt]&\leq\E\lt[\lal Z_t^{1,2}\ral^{\kappa}\rt]\\
&\leq 2^{\frac{\kappa}{2}+1}\E\lt[\lal X_t^{1}\ral^{\kappa}\rt]\leq 2^{\frac{\kappa}{2}+1}C_{a,\kappa,t},
\end{split}
\]
where $C_{a,\kappa,t}$ is the constant exhibited in Lemmma \ref{lem:mom}.
This and letting $\eta$ go to $0$ in \eqref{eq:regcomp} yields

\[
\frac{2^{\frac{\kappa}{2}+1}C_{a,\kappa,t}}{2\e\lt( c_{2,a}(\frac{\e\pi}{2(2-a)}-2\pi C_{\e,a}) - \chi\rt)}\geq\int_0^t\E\lt[\lt|Z_s^{1,2}\rt|^{\e-a}\rt]\,ds,
\]
provided that
\[
\frac{\e c_{2,a}\pi}{2(2-a)}-2\pi c_{2,a}C_{\e,a} >\chi.
\]
where we recall that $C_{\e,a}=\lt(\frac{2-\e}{\sqrt{4-\e}(3-a)\e}+\frac{1 }{a\e}\rt)$. Choosing $\e$ as close to $1$ as desired, by continuity the condition thus becomes

\bq
\label{eq:cond}
\frac{c_{2,a}\pi}{2(2-a)}-c_{2,a}2\pi\lt(\frac{1}{\sqrt{3}(3-a)}+\frac{1 }{a}\rt) >\chi.
\eq

Note that 
\begin{align*}
& \lim_{a\rightarrow 2^{-}}c_{2,a}=0, \\
& \lim_{a\rightarrow 2^{-}}\frac{c_{2,a}\pi}{2(2-a)}=\lim_{a\rightarrow 2^{-}}\frac{2^{a} a^2 \Gamma^{2}\lt(\frac{a}{2}\rt)\pi}{8\pi^2} \frac{\sin\lt(a \frac{\pi}{2}\rt)}{2-a}=1,
\end{align*}
since $\lim_{a\rightarrow 2^{-}} \frac{\sin\lt(a \frac{\pi}{2}\rt)}{2-a}=\frac{\pi}{2}$. Then we define $a^*$ such as
\[
a^*=\inf_{ \tilde{a} \in (1,2) } \Bigl\{ \forall a>\tilde{a}, \ \frac{c_{2,a}\pi}{2(2-a)}-c_{2,a}2\pi\lt(\frac{1}{\sqrt{3}(3-a)}+\frac{1 }{a}\rt)>0  \Bigr\}, 
\]
and for $a>a^*$, we define $\chi_a$ as
\[
\chi_a=\frac{c_{2,a}\pi}{2(2-a)}-c_{2,a}2\pi\lt(\frac{1}{\sqrt{3}(3-a)}+\frac{1 }{a}\rt),
\]
and the proof of Proposition \ref{prop:tiFC} is completed.
\subsubsection{Martingale method for convergence/consistency}
Using now Propositions \ref{prop:tiDD} and \ref{prop:tiFC}, we can give the

\begin{theorem}
	\label{thm:conv}
	Let be $T>0$ and either $1<\alpha<a<2$ and $\chi>0$ or $a^*<\alpha=a<2$ and $\chi\in (0,\chi_a)$. Consider $(X_t^i)_{t\in[0,T], i=1,\cdots,N}$ a sequence of solution to \eqref{eq:part_fr_KS}, and denote $\mu_t^N:=\frac{1}{N}\sum_{i=1}^N \delta_{X_t^i}$ the associated empirical measure. Assume that $F_0^N=\LL\lt( X_0^1,\cdots,X_0^N \rt)$ the law of the initial condition satisfies 
	\[
	F_0^N=\rho_0^{\otimes N}, \, \mbox{with} \, \rho_0\in L\log L(\R^2)\cap \PP_{\kappa}(\R^2), \, \text{if} \, a>\alpha, 
	\]
	or
	\[
	(F_0^N)_{N\geq 1}\, \mbox{is} \, \rho_0-\mbox{chaotic}, \,
	(F_0^N)_{N\geq 1}\in \PP_\kappa(\R^{2N}), \,  \text{if} \, a=\alpha\in (a^*,2), 
	\]
	for some $\kappa \in \lt(1,a\rt)$. Then
	\begin{itemize}
		\item[$(i)$] $(\mu^N_.)_{N\in \mathbb{N}}$ is tight in $\PP(\DD([0,T],\R^2))$
		\item[$(ii)$] any accumulation point of $(\mu^N_.)_{N\in \mathbb{N}}$ almost surely belongs to $\mathcal{S}_{\alpha}^a$ in case $a>\alpha$ or  $\mathcal{S}_{a}$ in case $a=\alpha$ respectfully defined as
		\bq
		\label{eq:deflimset}
		\begin{split}
			&\mathcal{S}_{\alpha}^a=\{ \rho\in \PP\lt(\DD\lt(0,T;\R^2)\rt) \rt) \, | \, \rho_.=\LL(\XX_.), \, (\XX_t)_{t\in[0,T]} \mbox{solution to \eqref{eq:NLSDE}, and } \, \int_0^T I_a(\rho_t)dt<\infty   \},\\
			&\mathcal{S}_{a}=\{ \rho\in \PP\lt(\DD\lt(0,T;\R^2)\rt) \rt)  \, | \, \rho_.=\LL(\XX_.), \, (\XX_t)_{t\in[0,T]} \mbox{solution to \eqref{eq:NLSDE}, and } \, \int_0^T\int_{\R^2\times \R^2} |x-y|^{\e-a}\rho_t(x)\rho_t(y)dxdydt<\infty \},\\ &\quad \forall \e\in(0,a) .
		\end{split}
		\eq
	\end{itemize}
	
\end{theorem}

\begin{proof}
	
	Proof of $(i)$ \newline
	First going back to \eqref{eq:ExpTi}, we deduce from Propositions \ref{prop:tiDD} and \ref{prop:tiFC} and a very standard argument (see \cite[Lemma 11]{FJ}, \cite[Lemma 5.2]{MSN} or \cite[Lemma 4.1]{GQ}) that  $\lt((X_0+J_t^{N,1})_{t\in[0,T]}\rt)_{N\geq 1}$ is tight in $\CC(0,T;\R^2)$.
	Hence $\lt((X_0+J_t^{N,1}+\ZZ_t^1)_{t\in[0,T]}\rt)_{N\geq 1}$ is tight in $\DD(0,T;\R^2)$. But using point $(ii)$ of Proposition \ref{prop:chaos} concludes the proof of point $(i)$ since $\DD([0,T],\R^2)$ is Polish (see \cite[Theorem 12.2]{Bil}). \newline	
	Proof of point $(ii)$ \newline
	Due to point $(i)$, we know that there is a subsequence of $(\mu_{.}^N)_{N\geq 1}$ (for which we will use the same notation for the sake of notational simplicity) going in law to some $\mu_.\in \PP\lt(\DD(0,T;\R^2)\rt)$. We now define the martingale problem of unknown $\mathcal{Q}\in \PP\lt(\DD(0,T;\PP(\R^2))\rt)$
	\bq
	\label{eq:mart}
	\left\{ \begin{array}{ll}
		\displaystyle (a) \quad \mathbf{e}_0\#\mathcal{Q}=\rho_0,&\\[2mm]
		\displaystyle(b)\quad  \mathcal{Q}_t:=\mathbf{e}_t\#\mathcal{Q}, (\mathcal{Q}_t)_{t\in[0,T]} \quad \text{satisfies \eqref{eq:deflimset}},&\\[2mm]
		\displaystyle (c) \quad \forall 0<t_1<\cdots<t_k<s<t\leq T, \phi_1,\cdots,\phi_k\in C_b(\R^2), \phi\in C^2_b(\R^2),\quad  \FF(\mathcal{Q})=0,  \mbox{with}&\\[2mm]
		\displaystyle \quad \quad \FF(Q):=&\\[2mm]
		\int \int \prod_{k=1}^N\phi_k(\gamma_{t_k}) \lt( \phi(\gamma_t)-\phi(\gamma_s)-\chi\int_s^tK_{\alpha}(\gamma_u-\tilde{\gamma}_u)  \cdot\nabla \phi(\gamma_u)-\mbox{v.p.}\int_{ \R^2 }\frac{\phi(\gamma_u+z)-\phi(\gamma_u)- z\cdot \nabla \phi(\gamma_u) }{|z|^{2+a}}dz \,du    \rt)Q(d\gamma)Q(d\tilde{\gamma}).
	\end{array} \right.
	\eq
	We now show that the limiting point $\mu_.\in \PP\lt(\DD(0,T;\R^2)\rt)$ solves this martingale problem, and divide it in 3 steps. \newline
	$\diamond$ Step $1$ \newline
	It is straightforward that in both cases $a=\alpha$ or $a>\alpha$, $\mu_.$ satisfies point $(a)$, due to the fact $F_0^N$ is $\rho_0$-chaotic. \newline
	$\diamond$ Step $2$\newline
	$\bullet$ Step $2.a$\newline 
	In the case $a=\alpha$ we use the techniques of \cite[\textit{Proof of Theorem 6, step 2.3}]{FJ} and introduce $m>0$. Due to Proposition \ref{prop:tiFC} we get
	\[
	\begin{split}
	\E\lt[ \int_0^T \int_{ \R^2\times \R^2 } (m \wedge |x-y|^{\e-a}) \mu_s^N(dx)\mu_s^N(dy)ds \ \rt]&=\frac{1}{N^2}\sum_{i,j}\E\lt[m \wedge |X_s^i-X_s^j|^{\e-a}\,ds\rt]\\
	&\quad\leq\frac{Tm}{N}+\frac{1}{N^2}\sum_{i\neq j} \int_0^T\E\lt[ |X_s^i-X_s^j|^{\e-a} \rt]\,ds\\
	&\quad \leq \frac{Tm}{N}+C_{\e,T}.
	\end{split}
	\] 
	Letting $N$ go to infinity, we find that the l.h.s. converges to $\E\lt[ \int_0^T \int_{ \R^2\times \R^2 } (m \wedge |x-y|^{\e-a}) \mu_s(dx)\mu_s(dy)\,ds \ \rt]$ since $\mu_.^N$ goes in law to $\mu_.$. Letting then $m$ go to infinity yields, thanks to the monotone convergence Theorem
	\[
	\E\lt[ \int_0^T \int_{ \R^2\times \R^2 }  |x-y|^{\e-a} \mu_s(dx)\mu_s(dy)ds  \rt]\,ds\leq C_{\e,T},
	\]
	and therefore
	\[
	\int_0^T \int_{ \R^2\times \R^2 }  |x-y|^{\e-a} \mu_s(dx)\mu_s(dy)ds<\infty, \, \mbox{a.s.}.
	\]. \newline
	$\bullet$ Step $2.b$ \newline	
	In case $a>\alpha$ we use the techniques of \cite[Proof of Proposition 6.1, Step 2]{MSN}. Denote $\pi_t=\LL(\mu_t)\in \PP_{\kappa}(\PP(\R^2))$, $\pi^j_t=\int_{  \PP(\R^2) } g^{\otimes j} \pi_t(dg)$ its Hewwit and Savage projection, and $F^{N,j}_t$ the marginal of $F_t^N$ on $\R^{2j}$ (recall that $F_t^N=\LL(X_t^1,\cdots,X_t^N)\in \PP_{\kappa}(\R^{2N})$). It is classical (see \cite{Szn}) to deduce from point $(i)$ that $F^{N,j}_t$ goes weakly to $\pi^j_t$ as $N$ goes to infinity. Hence, using Corollary \ref{lem:gam}, Fatou's Lemma and Proposition \ref{prop:tiDD} we find
	\[
	\begin{split}
	\E\lt[\int_0^T I_a(\mu_t)  \,dt\rt]&=\int_0^T \int_{  \PP(\R^2) } I_a(g)\pi_t(dg)dt \leq \int_0^T \liminf_{N\rightarrow +\infty} \II_a^N(F_t^N)dt\\
	&\quad \leq \liminf_{N\rightarrow +\infty} \int_0^T  \II_a^N(F_t^N)dt \leq 2\liminf_{N\rightarrow +\infty}\HH^{N}(F^N_0)+CT\\
	&\quad =2\liminf_{N\rightarrow +\infty}\HH^{N}(\rho_0^{\otimes N})+CT=2\int_{\R^2}\rho_0\ln \rho_0+CT,
	\end{split}
	\]
	and therefore
	\[
	\int_0^T I_a(\mu_t)  \,dt<\infty, \, \mbox{a.s.}
	\]
	$\diamond$ Step 3\newline
	Regardless of $a=\alpha$ or $a>\alpha$, define
	\[
	\begin{split}
	O_i(t)&:=\phi(X_t^i)-\int_0^t\frac{\chi}{N}\sum_{j\neq i}^N \nabla \phi(X_s^i)\cdot K_{\alpha}(X_s^i-X_s^j)ds\\
	&\quad-\int_{0}^t c_{2,a}\int_{\R^2 }\frac{\phi(X_{s}^i+z)-\phi(X_{s}^i)-z\cdot \nabla \phi(X_{s}^i) }{|z|^{2+a}}dzds \\
	&\quad= \phi(X_0^i)+\int_{[0,t]\times \mathbb{\R}^2}\lt(\phi(X_{s_-}^i+z)-\phi(X_{s_-}^i)\rt)\bar{M}_i(ds,dz).
	\end{split} 
	\]
	Note that $(O_t^i)_{t\in [0,T]}$ is a martingale with respect to the natural filtration of $\FF_t=\sigma(\XX_r^N,\ZZ_r^N)_{r\geq t}$. Hence since the $(M_i)_{i=1,\cdots,N}$ are independent, we have,
	\bq
	\label{eq:martdec}
	\begin{split}
		\E\lt[\FF^2(\mu_{.}^N)\rt]=&\E\lt[ \lt(\frac{1}{N}\sum_{i=1}^{N} \lt( \prod_{l=1}^k\phi_l(X^i_{t_l})\rt) \lt(O_i(t)-O_i(s)\rt) \rt)^2\rt]\\
		&\quad =\E\lt[\E\lt[ \frac{1}{N^2}\sum_{i,j}\prod_{l=1}^k\phi_l(X^i_{t_l})\prod_{l=1}^k\phi_l(X^j_{t_l})  \lt(O_t^i-O_s^i\rt)\lt(O_t^j-O_s^j\rt)\,| \, \FF_s\rt]\rt]\\
		&\quad =\frac{1}{N^2}\sum_{i=1}^N\E\lt[ \prod_{l=1}^k\phi_l(X^i_{t_l})^2 \rt]\E\lt[\lt( \int_{[s,t]\times \mathbb{\R}^2}\lt(\phi(X_{u_-}^i+z)-\phi(X_{u_-}^i)\rt)\bar{M}_i(du,dz)\rt)^2\rt]\leq \frac{C_{\FF}}{N},
	\end{split}
	\eq
	since (see for instance \cite{Ebe})
	\[
	\begin{split}
	\E\lt[\lt(\int_{[s,t[\times \R^2}\lt(\phi(X_{u_-}^i+z)-\phi(X_{u_-}^i)\rt)\bar{M}_i(du,dz)\rt)^2\rt]&=\int_s^t \E\lt[c_{2,a}\int_{  \R^2 }\frac{\lt(\phi(X_{u_-}^i+z)-\phi(X_{u_-}^i)\rt)^2}{|z|^{2+a}}dz\rt]du\\
	&\quad \leq c_{2,a}(t-s)\lt(\|\nabla \phi\|_{L^\infty} ^2 \int_{\mathbb{B}}|z|^{-a}dz + 4\|\phi\|^2_{L^\infty}\int_{  \mathbb{B}^c}|z|^{-(2+a)}dz   \rt).
	\end{split}	
	\]
	After which it is classical to deduce that $\mu_.$ satisfies $c)$ (see \cite[\textit{Proof of Theorem 6, Step 2.3.2-3-4}]{FJ}). Indeed for $\eta>0$ define $K_{\eta,\alpha}$ as
	\[
	K_{\eta,\alpha}(x)=-\frac{x}{\max(|x|^\alpha,\eta)^\alpha},
	\]
	and define $\FF_{\eta}$ as in \eqref{eq:mart} $c)$ with $K_\alpha$ replaced with $K_{\alpha,\eta}$. Note that with the non self interaction condition it holds for any $x\in \R^2$ and $\e\in (0,1)$
	\[
	\lt|K_{\alpha}(x)-K_{\alpha,\eta}(x) \rt|\leq \mb_{|x|\leq \eta}|x|^{1-\alpha}	\leq \eta^{1-\e} |x|^{\e-a}.
	\]
	This deduces that for any $Q\in\PP(\DD(0,T;\R^2))$
	\bq
	\label{eq:regfunc}
	\begin{split}
		\lt|\FF(Q)-\FF_\e(Q) \rt|&\leq \eta^{1-\e} \int \int \int_0^t |\gamma_u-\tilde{\gamma}_u|^{\e-a}Q(d\gamma)Q(d\tilde{\gamma})\\
		&\quad \leq  \eta^{1-\e}  \int_0^t\int \int  |z-\tilde{z}|^{\e-a}Q_u(dz)Q_u(d\tilde{z})\,du.
	\end{split}
	\eq
	Hence this deduces
	\[
	\begin{split}
	\E\lt[\lt|\FF(\mu_.)\rt|\rt]&\leq \E\lt[\lt|\FF(\mu_.)-\FF_{\eta}(\mu_.)\rt|\rt]+\limsup_{N} \lt|\E\lt[\lt|\FF_{\eta}(\mu_.)\rt|\rt]-\E\lt[\lt|\FF_{\eta}(\mu^N_.)\rt|\rt]\rt|\\
	&\quad + \limsup_{N} \E\lt[\lt|\FF_{\eta}(\mu^N_.)-\FF(\mu^N_.)\rt|\rt]+\limsup_{N} \E\lt[\lt|\FF(\mu^N_.)\rt|\rt]\\
	&\quad := I_1+I_2+I_3+I_4.
	\end{split}
	\]
	$I_2$ goes to $0$ when $N$ goes to infinity and $\eta$ is fixed since $\FF_{\eta}$ is a smooth function on $\DD(0,T;\R^2)$ and $\mu_.^N$ goes in law to $\mu_.$. $I_4$ goes to $0$ when $N$ goes to infinity due to \eqref{eq:martdec}. $I_1$ and $I_3$ go to $0$ as $\eta$ goes to $0$ due to \eqref{eq:regfunc} and respectively Step 2 of this proof and Proposition \ref{prop:tiDD} or Proposition \ref{prop:tiFC} depending on whether $a=\alpha$ or $a>\alpha$. Finally we deduce
	\[
	\E\lt[\lt|\FF(\mu_.)\rt|\rt]=0,
	\]
	which implies $\FF(\mu_.)=0$ a.s. and concludes the proof.
	
\end{proof}

%
%
%
%
\section{Uniqueness of the limit equation}
Now, in order to complete the propagation of chaos result, we need to investigate the uniqueness of the accumulation points of the sequence of the law of solution to equation \eqref{eq:part_fr_KS}, which have been proved to be tight in both case $a>\alpha$ and $a=\alpha$ in the previous section. However this uniqueness can not be obtained in the Fair Competition case in the class where lie the accumulation points. However in the Diffusion Dominated we are able to conclude to the well posedness of equation \eqref{eq:Fr_KS} for an initial condition in $L\log L(\R^2)$.

From Lemmma \ref{lem:GNS}, we know that any solution to \eqref{eq:Fr_KS} starting from an $L\log L$ initial condition is $L^1(0,T;L^p(\R^2))$ for any $p\in \lt(1, \frac{2}{2-a}\rt)$. It is now to prove the uniqueness of solution to equation \eqref{eq:Fr_KS} in this class. Note that this is known in the case $a=2$ and $\alpha\in(1,2)$ (see \cite{GQ}). Precisely we have strong-strong stability estimate 
	
\begin{lemma}
	\label{lem:log_lip}
	Let be $p>\frac{2}{2-\alpha},q\geq 1$. There is a constant $C_{\alpha,q,p}>0$ such that for any $X,Y$ random variables of respective laws $\rho_1,\rho_2$ with $\rho_1,\rho_2\in L^{p}(\R^2)\cap \PP_q(\R^2)$. Then it holds
	\[
	\E\lt[\lt|\int_{\R^2} K_{\alpha}(X-y)\rho_1(dy)-\int_{\R^2} K_{\alpha}(Y-y)\rho_2(dy)\rt||X-Y|^{q-1}\rt]\leq C_{\alpha,q,p}\lt(2+\|\rho_1\|_{ L^p}+\|\rho_2\|_{L^p}\rt)\E[|X-Y|^q].
	\]
	
\end{lemma}

\begin{proof} First, we notice that
	\[
	\nabla K_{\alpha}(x)=\lt(I_2-\alpha \frac{x\otimes x}{|x|^2}\rt)|x|^{-\alpha},
	\]
	hence, denoting $\bar{X},\bar{Y}$ some independent copies of $X,Y$, we find
	\[
	|K_{\alpha}(X-\overline{X})-K_{\alpha}(Y-\overline{Y})| \leq C_{\alpha} \lt(|X-Y|+|\overline{X}-\overline{Y}|\rt)\lt(\frac{1}{|X-\overline{X}|^\alpha}+\frac{1}{|Y-\overline{Y}|^\alpha}\rt).
	\]
	This yields to
	$$\begin{aligned}
	\E\lt[|K_{\alpha}(X-\overline{X})-K_{\alpha}(Y-\overline{Y})||X-Y|^{q-1} \rt]&\leq C\E\lt[\lt(\frac{1}{|X-\overline{X}|^{\alpha}}+\frac{1}{|Y-\overline{Y}|^{\alpha}}\rt)|X-Y|^{q}\rt]\\
	& +C\E\lt[\lt(\frac{1}{|X-\overline{X}|^{\alpha}}+\frac{1}{|Y-\overline{Y}|^{\alpha}}\rt)|\overline{X}-\overline{Y}||X-Y|^{q-1}\rt]\\
	&=: C(I_1+I_2).
	\end{aligned}$$
	$\diamond$ Estimate of $I_1$: 	First we easily get for this term, by taking firstly the expectation on $(\overline{X},\overline{Y})$ 
	$$
	\begin{aligned}
	I_1=&\E_{X,Y}\lt[|X-Y|^q\E_{\overline{X},\overline{Y}}\lt[\lt(\frac{1}{|X-\overline{X}|^{\alpha}}+\frac{1}{|Y-\overline{Y}|^{\alpha}}\rt)\rt ]\rt]\\
	& \quad \leq  C_{\alpha,p}\lt(\|\rho_1\|_{L^p(\R^2)}+\|\rho_2\|_{L^p(\R^2)}+\|\rho_1\|_{L^1(\R^2)}+\|\rho_2\|_{L^1(\R^2)}\rt)\E\lt[|X-Y|^q\rt].
	\end{aligned}
	$$
	
	$\diamond$ Estimate of $I_2$: 
	We use Holder's inequality to find
	$$\begin{aligned}
	I_2&\leq C\E\lt[|\overline{X}-\overline{Y}|\lt(\frac{1}{|X-\overline{X}|^{\alpha}}+\frac{1}{|Y-\overline{Y}|^{\alpha}}\rt)^{1/q}\lt(\lt(\frac{1}{|X-\overline{X}|^{\alpha}}+\frac{1}{|Y-\overline{Y}|^{\alpha}}\rt)^{1/q}|X-Y|\rt)^{q-1}\rt]\\
	& \leq C\E\lt[|\overline{X}-\overline{Y}|^q\lt(\frac{1}{|X-\overline{X}|^{\alpha}}+\frac{1}{|Y-\overline{Y}|^{\alpha}}\rt)\rt]^{1/q}\E\lt[\lt(\frac{1}{|X-\overline{X}|^{\alpha}}+\frac{1}{|Y-\overline{Y}|^{\alpha}}\rt)|X-Y|^{q}\rt]^{(q-1)/q}\\
	&=C\E\lt[|\overline{X}-\overline{Y}|^q\lt(\frac{1}{|X-\overline{X}|^{\alpha}}+\frac{1}{|Y-\overline{Y}|^{\alpha}}\rt)\rt].
	\end{aligned}$$
	Now taking first the expectation w.r.t. $(X,Y)$ and then w.r.t. $(\overline{X},\overline{Y})$ yields similarly as above
	$$I_2\leq C_{\alpha,q}\lt(\|\rho_1\|_{L^1\cap L^p(\R^2)}+\|\rho_2\|_{L^1\cap L^p(\R^2)}\rt) \E\lt[ |X-Y|^q\rt].$$
	Hence putting all those estimates together leads to the desired result.

\end{proof}

Then we can obtained the desired stability estimate stated in the
\begin{proposition}
	\label{prop:stab}
	Let be $T>0,q\in (1,a)$, and $(\XX^1_t)_{t \in [0,T]}$ and $(\XX^2_t)_{t \in [0,T]}$ two solutions to equation \eqref{eq:NLSDE} build with the same L\'evy process, and assume their respective laws $\rho_.^1,\rho_.^2\in L^1(0,T;L^p(\R^2))\cap \PP\lt(\DD(0,T;(\R^2))\rt)$ for some $p>\frac{2}{2-\alpha}$. Then it holds
	\[
	\E\lt[|\XX_t^1-\XX_t^2|^q \rt]\leq \E\lt[|\XX_0^1-\XX_0^2|^q \rt]e^{2t+\int_0^t\lt( \|\rho_s^1\|_{L^p}+ \|\rho_s^1\|_{L^p}  \rt)\,ds }.
	\] 
\end{proposition}

\begin{proof}
Since the two solutions are build on the same L\'evy process, we get
\[
\begin{split}
|\XX_t^1-\XX_t^2|^q &=|\XX_0^1-\XX_0^2|^q+q\int_0^t |\XX_s^1-\XX_s^2|^{q-1} \lt| \int_{ \R^{2} }K_{\alpha}(\XX^1_s-y)\rho^1_s(dy)-\int_{ \R^{2} }K_{\alpha}(\XX^2_s-y)\rho^2_s(dy)   \rt|\,ds.
\end{split}
\]
Taking the expectation, using Lemma \ref{lem:log_lip} and Gronwall's inequality yields the desired result.

\end{proof}

\begin{corollary}
	\label{cor:uni}
	When $a>\alpha$, the set $S_{\alpha}^a$ defined in \eqref{eq:deflimset} is a singleton.
\end{corollary}

\begin{proof}
	Recall that
	\[
	\mathcal{S}_{\alpha}^a=\{ \rho\in \PP\lt(\DD\lt(0,T;\R^2\rt)\rt)  \, | \, \rho_.=\LL(\XX_.), \, (\XX_t)_{t\in[0,T]}\, \mbox{solution to \eqref{eq:NLSDE}, and } \, \int_0^T I_a(\rho_t),\ dt<\infty   \}.
	\]
	is not empty due to Theorem \ref{thm:conv}. But due to Lemma \ref{lem:GNS}, if $\rho\in \mathcal{S}_{\alpha}^a$ then $\rho\in L^{1}\lt(0,T; L^{q}(\R^2) \rt)$,
	for any $q\in \lt(1,\frac{2}{2-a}\rt)$. But since $\frac{2}{2-a}>\frac{2}{2-\alpha}$, one can choose $q\in\lt(\frac{2}{2-\alpha},\frac{2}{2-a}\rt)$ and due to Proposition \ref{prop:stab} there is at most one solution to equation \eqref{eq:NLSDE} with initial condition of law $\rho_0\in \PP_\kappa(\R^2)$.
\end{proof}	
Theorem \ref{thm:conv} and Corollary \ref{cor:uni} conclude the proof of Theorem \ref{thm:main}.

\appendix
\section{On the critical sensitivity in the Fair Competition case}

It is possible, in the Fair Competition case, to give a more precise result than the one provided in Theorem \ref{thm:main}, at the cost of a less rigorous proof. Let us briefly recall that it is possible to give an alternate definition of $-(-\Delta)^{a/2}$ via Fourier's analysis (see for instance \cite{K}), which enables to obtain

\bq
\label{eq:fralapexa}
-(-\Delta)^{a/2}(|x|^\e)= -2^a\frac{\Gamma\lt(\frac{2+\e}{2}\rt)\Gamma\lt(\frac{a-\e}{2}\rt)}{\Gamma\lt(-\frac{\e}{2}\rt)\Gamma\lt(\frac{2+a-\e}{2}\rt)} |x|^{\e-a}. 
\eq
Applying Ito's rule to $Z_s^{i,j}$ with $\phi(x)=|x|^\e$ for some $\e\in(0,1)$ is not possible, since the $\phi$ defined so is not $\CC^2$ (not even $\CC^a$), but let us perform the computations for the sake of the discussion. Coming back to \eqref{eq:ITO} and taking the expectation formally yields
\bq
\label{eq:ITOfom}
\begin{aligned}
\E\lt[\lt|Z_t^{i,j}\rt|^\e\rt]&=\E\lt[\lt|Z_t^{i,j}\rt|^\e\rt]-\int_0^t \frac{\chi}{N} \e \E\lt[|Z_s^{i,j}|^{\e-2}Z_s^{i,j}\cdot \lt(\sum_{k\neq i,j}\lt( \frac{Z_s^{i,k}}{|Z_s^{i,k}|^a}-\frac{Z_s^{j,k}}{|Z_s^{j,k}|^a}\rt)\rt)\rt]\, ds\\
&\quad- \int_0^t \frac{2\chi\e}{N} \E\lt[|Z_s^{i,j}|^{\e-a}\rt]\,ds+2\int_0^t\E\lt[(-\Delta)^{a/2}\phi (Z_s^{i,j})\rt] \,ds.
\end{aligned}
\eq
Recalling \eqref{eq:fralapexa} yields
\[
\begin{split}
\E\lt[\lt|Z_t^{i,j}\rt|^\e\rt]&\geq\E\lt[\lt|Z_t^{i,j}\rt|^\e\rt]- 2\e\lt(\chi+2^a\frac{\Gamma\lt(\frac{2+\e}{2}\rt)\Gamma\lt(\frac{a-\e}{2}\rt)}{\Gamma\lt(-\frac{\e}{2}\rt)\Gamma\lt(\frac{2+a-\e}{2}\rt)\e}\rt) \int_0^t \E\lt[|Z_s^{i,j}|^{\e-a}\rt]\,ds
\end{split}
\]
So that, optimizing w.r.t. $\e\in(0,1)$, the condition becomes
\[
\chi<\chi_a,
\]
with
\begin{align*}
\chi_a=&-\sup_{\e\in (0,1)}2^a\frac{\Gamma\lt(\frac{2+\e}{2}\rt)\Gamma\lt(\frac{a-\e}{2}\rt)}{\Gamma\lt(-\frac{\e}{2}\rt)\Gamma\lt(\frac{2+a-\e}{2}\rt)\e}\\
&=\sup_{\e\in (0,1)}2^{a-2}\frac{\e\sin (\e\pi/2)\Gamma^2(\e/2)}{\pi(a-\e)}.
\end{align*}
Since this appendix is rather formal, we bound ourselves to numerically check that this supremum is obtained as $\e\rightarrow 1^-$ which yields
\[
\chi_a=\frac{2^{a-2}}{a-1}, \ a^*=1.
\]
Nevertheless it seems not obvious how to rigorously prove this threshold. Indeed it would seem strange that any argument could enable to let $\eta$ go to $0$ in \eqref{eq:ITO} in order to obtain \eqref{eq:ITOfom}, as it would consist in a justification of the use of the Ito's rule with a not enough regular function. It is thus a necessity to perform the computations with $\phi_\eta(x)=\eta^\e\lal \frac{x}{\eta}\ral^\e$, for which the explicit formula of the fractional Laplacian is not known (in the classical case, $\Delta \phi_\eta(x)=\e(|x|^2+\eta^2)^{\e/2-2}(\e|x|^2+2\eta^2)$, and it does not change much to perform the computations with $\phi$ or $\phi_\eta$ ). Therefore we are compelled to give in Lemma \ref{lem:frlap} a bound by below of $-(-\Delta \phi_\eta)^{a/2}$ with some nice dependence on $\eta$, which is rather rough.
\section{Remarks on the case $\alpha \leq 1$}

As mentioned in the introduction, in this case the aggregation kernel $K_\alpha$ is less singular, given that it is $1-\alpha$-Holder. It would be interesting to obtain some quantitative result, similarly as \cite{Hol} in the case $a=2$. Stating a convergence/consistency result, in the case $\alpha\leq 1$ is rather cheap, given that it consists mostly in a control moment. Nevertheless we give the outlines of the proof as conclusion remark. Indeed let be $(a,\alpha)\in(0,2)\times(0,1)$, coming back to \eqref{eq:ExpTi}, we need to bound for some $p>1$, uniformly in $N$ the following quantity
\begin{align*}
	\E\lt[Z_{N,p}^T\rt]=& 1+\frac{1}{N}\sum_{j>1} \int_0^T\E\lt[
	|X_u^{1}-X_u^j|^{(1-\alpha)p}\rt]\,du\\
	&\leq 1+\int_0^T2^{(1-\alpha)p-1}\lt(\E\lt[|X_u^1|^{(1-\alpha)p}\rt]+\frac{1}{N}\sum_{j>1}^N \E\lt[|X_u^j|^{(1-\alpha)p}\rt]\rt)du\\
	&\leq 1+2^{(1-\alpha)p}\int_0^T\int\lal x_1\ral^{(1-\alpha)p} F_t^N dt.
\end{align*}
The bound is straightforward in the case $\alpha=1$. Otherwise it is a matter of propagation of moments of order $(1-\alpha)p$ for some $p>1$ in order to obtain the proof of point $(i)$ of Theorem \ref{thm:conv}, and then straightforwardly follow the proof of point $(ii)$. But due to \eqref{eq:momfralap} this is possible only if $(1-\alpha)p<a$ i.e. $1<a+\alpha$. Therefore we can expect a convergence/consistency result in the all area 
\[
\{ (a,\alpha)\in(0,2)\times(0,1], \ 1<a+\alpha \}.  
\]
Then we divide this area in three (see Figure \ref{fig1}). 
\begin{itemize}
	\item$(a,\alpha)\in (1,2)\times(0,1)$\newline 
	Making the assumption $F_0^N=\rho_0^{\otimes N}$ with $\rho_0\in L\ln L\cap \PP_\kappa(\R^2)$ with $\kappa \in (1,a)$ similarly as in the Diffusion Dominated case in Theorem \ref{thm:main}, Lemma \ref{lem:mom} apply and we let $p=\frac{\kappa}{1-\alpha}>1$. And we obtain the claimed convergence/consistency result.\newline
	Then since Lemmas \ref{lem:EntDis} and \ref{lem:log_lip} also apply, we can deduce to the uniqueness of the limit point and obtain a complete propagation of chaos result.
	\item$(a,\alpha)\in (0,1)\times(0,1), 1<a+\alpha$\newline 
	In this case it is possible to obtain the bound
	\[
	\int_0^T\sup_{N\geq 1}\int_{\R^{2N}}\lal x_1\ral ^{\kappa}F_t^N dt <\infty,
	\]  
	for some $\kappa\in (1-\alpha, a)$, but the moment estimate is more complicated than in the proof of Lemma \ref{lem:mom}, as for $\kappa<1$, $m_\kappa$ is not convex, and one can not enjoy point $(iii)$ of Lemma \ref{lem:use1}. We do not treat this problem here. Nevertheless, should this bound be obtained, it would immediately imply the tightness of the sequence of the empirical measure, and the convergence of a subsequence to some element of the set
	\[
	\mathcal{S}_{\alpha}^a=\{ \rho\in \PP\lt(\DD\lt(0,T;\R^2\rt)\rt)  \, | \, \rho_.=\LL(\XX_.), \, (\XX_t)_{t\in[0,T]}\, \mbox{solution to \eqref{eq:NLSDE}, and } \, \int_0^T \int_{\R^{2}}|x|^{\kappa}\rho_t,\ dt<\infty   \}.
	\]
	\begin{itemize}
		\item $\alpha<a$ \newline
		This class is too large to state a uniqueness result given the lack of regularity of $K_\alpha$. It could be interesting to try to extend the result of Corollary  \ref{lem:gam} in the range $a\in (0,1)$, in order to gain some regularity on the limiting point, by passing some fractional Fisher information of low order to the limit.
		\item $\alpha\geq a$ \newline
		Here we can not go beyond the convergence/consistency result.
	\end{itemize} 
	
	 Otherwise, mollifying the kernel $K_\alpha$ near the origin so that it is Lipschitz, uniqueness can be obtained by standard coupling arguments. However, note that the above strategy yields the existence of a solution to the nonlinear equation \eqref{eq:Fr_KS}, with $K_\alpha$ replaced by its mollification. In the classical case (see \cite{Szn}), this existence is usually proved by a fix point argument in $\CC([0,T],\PP_k(\R^d))$ for some $k>1$. Here we could not use this strategy given that we can not expect on the solution moment of higher order than $a\leq 1$.

\end{itemize}

\section{Proof of Proposition  \ref{lem:aff} }
We begin this section by defining for $\e>0$, $\psi_\e$ on $\R^2$ as
\[
\psi_\e(x)=\e^{-2}e^{-\sqrt{1+\lt(\frac{|x|}{\e}\rt)^2}}.
\]
which is borrowed from \cite{HM}. Observe that for any $x,y\in \R^2\times\R^2$ it holds
\bq
\label{eq:loin}
\begin{aligned}
\Phi(\psi_\e(x),\psi_\e(y))&=-(\psi_\e(x)-\psi_\e(y))\lt( \sqrt{1+\lt(\frac{|x|}{\e}\rt)^2}-\sqrt{1+\lt(\frac{|y|}{\e}\rt)^2} \rt)\\
&\leq \e^{-1} |x-y|\lt( \psi_\e(x)+\psi_\e(y) \rt)
\end{aligned}
\eq

\begin{lemma}
	\label{lem:reg}

	Let be $\pi\in \PP(\PP(\R^2))$ and for $N\geq 1$ define
	\[
	\pi^{N,\e}=\int_{\PP(\R^2)}(\rho * \psi_\e)^{\otimes N}\pi(d\rho).
	\]
	Then for any $x_2,\cdots,x_N:=X^{N-1}\in \R^{2(N-1)}$, 
	define $p^\e(\cdot|X^{N-1})$ the conditional law knowing $X^{N-1}$ under $\pi^{N,\e}$. Then it holds
	\begin{itemize}
		\item[$(i)$] 
		\[
		\|\nabla \ln p^\e(\cdot|X^{N-1})\|_{L^\infty}	\leq \e^{-1}.
		\]
		\item[$(ii)$] there exists a constant $C_{\e,R}>0$ such that for any $x\in \R^2,\mathbf{e}\in \mathbb{S}^1$ and $u\in [0,R]$
		\[
		p^\e(x+u\mathbf{e}|X^{N-1})\leq C_{\e,R} p^\e(x|X^{N-1}).
		\]
		\item[$(iii)$]
		\[
		\Phi(\pi^{N,\e}(X^N_x),\pi^{N,\e}(X^{N}_y))\leq \pi^{N-1,\e}(X^{N-1})\Phi(\psi_\e(x),\psi_\e(y))
		\]
	\end{itemize}

\end{lemma}

\begin{proof}
	Proof of $(i)$: \newline
	First note that
	\[
	p^\e(x|X^{N-1})=\frac{\pi^{N,\e}(x,x_2,\cdots,x_N)}{\pi^{N-1,\e}(x_2,\cdot,x_N)},
	\]
	Indeed due to Fubini's Theorem one can check that
	\[
	\begin{split}
	\int_{\R^2}\pi^{N,\e}(x,x_2,\cdots,x_N)dx&=\int_{\R^2}\int_{\PP(\R^2)}(\rho * \psi_\e)(x)\prod_{k=2}^N(\rho * \psi_\e)(x_k)\pi(d\rho)dx\\
	&\quad = \int_{\PP(\R^2)} \lt(\int_{\R^2}(\rho * \psi_\e)(x)dx\rt)\prod_{k=2}^N(\rho * \psi_\e)(x_k)\pi(d\rho)=\pi^{N-1,\e}(x_2,\cdots,x_N),
	\end{split}
	\]
	and the sequence $(\pi^{N,\e})$ is compatible. Hence
	\[
	\nabla \ln p^\e(\cdot|X^{N-1})=\nabla_1 \ln \pi^{N,\e}(\cdot,x_2,\cdots,x_N),
	\]
	and we use \cite[Lemma 5.9]{HM} to conclude the result. \newline
	
	Proof of $(ii)$: \newline
	Observe that for any $x\in \R^2, \mathbf{e}\in \mathbb{S}^1$ and $u\in [0,R]$ it holds
	\begin{align*}
			\lt|\sqrt{1+\frac{|x+u\mathbf{e}|^2 }{\e^2}}- \sqrt{1+\frac{|x|^2}{\e^2}}\rt|&=\e^{-2}\frac{\lt||x+u\mathbf{e}|^2-|x|^2\rt|}{\sqrt{1+\frac{|x+u\mathbf{e}|^2 }{\e^2}}+ \sqrt{1+\frac{|x|^2}{\e^2}}}\\
			&\leq \e^{-1}u\frac{\e^{-1}|x+u\mathbf{e}|+\e^{-1}|x|}{\sqrt{1+\frac{|x+u\mathbf{e}|^2 }{\e^2}}+ \sqrt{1+\frac{|x|^2}{\e^2}}}\leq \e^{-1}R, 
		\end{align*}	
		hence
		\begin{align*}
			\sqrt{1+\frac{|x+u\mathbf{e}|^2 }{\e^2}}\geq \sqrt{1+\frac{|x|^2}{\e^2}}-\e^{-1}R, \ \text{i.e.} \ \psi_\e(x+\mathbf{e}u)\leq e^{\e^{-1}R}\psi_\e(x),
		\end{align*}
		so that	
		
	\begin{align*}
	\rho*\psi_\e(x+u\mathbf{e})=\int \psi_\e(x+u\mathbf{e}-y) \rho(dy)\leq C_{\e,R}\int \psi_\e(x-y) \rho(dy)=C_{\e,R}\rho*\psi_\e(x)
	\end{align*}
	Proof of $(iii)$:\newline
	By convexity of $\Phi$ and Jensen's inequality we successively obtain
	\begin{align*}
	\Phi(\pi^{N,\e}(X^N_x),\pi^{N,\e}(X^{N}_y))&=\Phi\lt( \int_{\PP(\R^2)} \rho*\psi_\e(x) \prod_{i=2}^N\rho*\psi_\e(x_i)  \pi(d\rho),\int_{\PP(\R^2)} \rho*\psi_\e(y) \prod_{i=2}^N\rho*\psi_\e(x_i)  \pi(d\rho) \rt)	\\
	&\leq \int_{\PP(\R^2)} \Phi\lt(\rho*\psi_\e(x) \prod_{i=2}^N\rho*\psi_\e(x_i),\rho*\psi_\e(y) \prod_{i=2}^N\rho*\psi_\e(x_i)  \rt)\pi(d\rho)\\
	&\leq \int_{\PP(\R^2)} \prod_{i=2}^N\rho*\psi_\e(x_i) \Phi\lt(\rho*\psi_\e(x) ,\rho*\psi_\e(y) \rt)\pi(d\rho)\\
	&\leq \int_{\PP(\R^2)} \prod_{i=2}^N\rho*\psi_\e(x_i) \Phi\lt(\psi_\e(x) ,\psi_\e(y) \rt)\pi(d\rho)=\pi^{N-1,\e}(X^{N-1})\Phi\lt(\psi_\e(x) ,\psi_\e(y) \rt)
	\end{align*}	
\end{proof}	

Before completing the proof we will furthermore use the following consideration. Let $F^N,G^N\in \PP(\R^{2N})$ and accordingly to the notations previously introduced for $x,y,x_2,\cdots,x_N\in \R^{2(N+1)}$ denote $X^x_N=(x,x_2,\cdots ,x_N)\in\R^{2N}$. Then straightforward computations yields
\[
\begin{split}
&\\
&\mathcal{D}^N:=\theta\Phi(F^N(X^x_N),F^N(X^y_N))+(1-\theta)\Phi(G^N(X^x_N),G^N(X^y_N))\\
&-\Phi(\theta F^N(X^x_N)+(1-\theta)G^N(X^x_N),\theta F^N(X^y_N)+(1-\theta)G^N(X^y_N))\\
&\quad =-\theta \lt( F^N(X^x_N)-F^N(X^y_N) \rt)\lt( \ln \lt( \frac{\theta F^N(X^x_N)+(1-\theta)G^N(X^x_N)}{ F^N(X^x_N)} \rt)- \ln \lt( \frac{\theta F^N(X^y_N)+(1-\theta)G^N(X^y_N)}{ F^N(X^y_N)} \rt)\rt)\\
&\quad -(1-\theta) \lt( G^N(X^x_N)-G^N(X^y_N) \rt)\lt( \ln \lt( \frac{\theta F^N(X^x_N)+(1-\theta)G^N(X^x_N)}{ G^N(X^x_N)} \rt)- \ln \lt( \frac{\theta F^N(X^y_N)+(1-\theta)G^N(X^y_N)}{ G^N(X^y_N)} \rt)\rt).
\end{split}
\]  
Note that due to the convexity of $\Phi$, $\DD^N$ is always nonnegative. Denote $f(\cdot|X^{N-1})$ (resp. $g(\cdot|X^{N-1})$) the conditional law w.r.t. to the first component knowing the last $N-1$ under $F^N$ (resp. $G^N$) i.e.
\begin{align*}
&F^N(x,x_2,\cdots,x_N)=f(x|x_2,\cdots,x_N)F^{N-1}(x_2,\cdots,x_N)\\
&G^N(x,x_2,\cdots,x_N)=g(x|x_2,\cdots,x_N)G^{N-1}(x_2,\cdots,x_N),
\end{align*} 
and define
\[
h(t)= \ln \lt( \theta+ (1-\theta)\frac{g(z_t|.)}{f(z_t|.)}\frac{G^{N-1}}{F^{N-1}}  \rt), \ z_t=tx+(1-t)y
\]
Since $f\nabla \frac{g}{f}=g \nabla \ln \frac{g}{f}$
\begin{align*}
	h'(t)&=\frac{ G^{N-1} f(z_t|.)\nabla \frac{g(|.)}{ f(|.)}(z_t)\cdot (x-y) }{\theta f(z_t|.)F^{N-1} +(1-\theta)g(z_t|.)G^{N-1}} \\
	&=\frac{ G^{N-1} g(z_t|.)\nabla \ln\frac{g(|.)}{ f(|.)}(z_t)\cdot (x-y) }{\theta f(z_t|.)F^{N-1} +(1-\theta)g(z_t|.)G^{N-1}},
\end{align*}

we can rewrite
\bq
\begin{aligned}
	\label{eq:id}
\mathcal{D}^N&=-\int_0^1 \frac{ (1-\theta)\theta G^{N-1}F^{N-1}\lt( f(x|)-f(y|) \rt) g(z_t|.)\nabla \ln\frac{g(|.)}{ f(|.)}(z_t)\cdot (x-y) }{\theta f(z_t|.)F^{N-1} +(1-\theta)g(z_t|.)G^{N-1}}dt\\
&\quad -\int_0^1 \frac{ (1-\theta)\theta G^{N-1}F^{N-1}\lt( g(x|)-g(y|) \rt) f(z_t|.)\nabla \ln\frac{f(|.)}{ g(|.)}(z_t)\cdot (x-y) }{\theta f(z_t|.)F^{N-1} +(1-\theta)g(z_t|.)G^{N-1}}dt.
\end{aligned}
\eq

	We are now in position to prove Lemma \ref{lem:aff}. Following the idea of \cite[\textit{Proof of Lemma 4.2}]{MSN} and \cite[\textit{Proof of Lemma 5.10}]{HM}, we only treat the case $M=2$ and $\omega_1=\mathcal{B}_r:=\{ \rho\in \PP_1(\R^{2})\,|\, W_1(\rho,f_1) < r \}$ for some $r>0$ and $f_1\in \PP_1(\R^{2})$. For $\pi \in \PP_\kappa(\PP(\R^2))$, define
		\[
		\theta:=\pi\lt(\omega_1\rt)(>0 \quad \text{w.l.o.g.}), \quad F:=\theta^{-1}\mb_{\omega_1}\pi, \, G:=(1-\theta)^{-1}\mb_{\omega^c_1}\pi.
		\]
		Our aim is to prove that
		\[
		\tilde{\II}_a(\pi)=\theta\tilde{\II}_a(F)+(1-\theta)\tilde{\II}_a(G),
		\]
		or equivalently, by convexity, that for any fixed $\eta>0$ 
		\[
		\theta\tilde{\II}_a(F)+(1-\theta)\tilde{\II}_a(G)-\tilde{\II}_a(\pi)<\eta.
		\] 
		Let then be $\eta>0$ fixed for the rest of the proof and for $N\geq 1,\e>0$ define 
		\[
		F^{N,\e}:=\int_{\PP_1(\R^{2})}(\rho*\psi_\e)^{\otimes N} F(d\rho), \, G^{N,\e}:=\int_{\PP_1(\R^{2})}(\rho*\psi_\e)^{\otimes N} G(d\rho)
		\]
		Also note that
		\[
		\pi^{N,\e}:=\int_{\PP_1(\R^{2})}(\rho*\psi_\e)^{\otimes N} \pi(d\rho)=\theta F^{N,\e}+(1-\theta)G^{N,\e},
		\]
		since $F$ and $G$ have disjunct supports. It is also clear (see for instance the computations done in the proof of Lemma \ref{lem:reg}), that the sequences $(\pi^{N,\e})_{N\geq 1}$, $(F^{N,\e})_{N\geq 1}$ and $(G^{N,\e})_{N\geq 1}$ are compatible, and denote $\pi^\e,F^\e$ and $G^\e$ in $\PP_{\kappa}(\PP(\R^2)))$ the probability measures which are associated to these sequences by the Hewitt and Savage Theorem.
		\begin{equation}
		\begin{split}
		\label{eq:aff}
		\mathcal{K}^N&:= \theta\II^N_a(F^{N})+(1-\theta)\II^N_a(G^{N})-\II^N_a(\pi^{N})\\
		&=\int_{\R^{2(N+1)}} |x-y|^{-(2+a)}  \lt(\theta \Phi(F^N(X_x^N),F^N(X_y^N)) + (1-\theta) \Phi(G^N(X_x^N),G^N(X_y^N)) - \Phi(\pi^N(X_x^N),\pi^N(X_y^N)) \rt). 
		\end{split}
		\end{equation}
		Let be $R>0$ such that $2\int_{\B^c_R}|x|^{-(1+a)}dx<\eta$. For couples $(x,y)\in \R^4$ such that $|x-y|\leq R$ we use the upper bound provided in \eqref{eq:id} and on the complementary set, the one of \eqref{eq:loin}. Which gives
		\begin{equation}
		\begin{split}
		\label{eq:aff}
		\mathcal{K}^N&\leq\int_{\R^{2(N-1)}} \int_{|x-y|\geq R}|x-y|^{-(2+a)}  \lt(\theta \Phi(F^N(X_x^N),F^N(X_y^N)) + (1-\theta) \Phi(G^N(X_x^N),G^N(X_y^N)) \rt)dxdydX^{N-1} \\	
		&+ \int_{\R^{2(N-1)}} \int_{|x-y|\leq R}|x-y|^{-(2+a)}  \lt(\theta \Phi(F^N(X_x^N),F^N(X_y^N)) + (1-\theta) \Phi(G^N(X_x^N),G^N(X_y^N)) - \Phi(\pi^N(X_x^N),\pi^N(X_y^N)) \rt)\\
		&\quad \leq \int_{\R^{2(N-1)}} \lt(\theta F^{N-1}(X^{N-1}) + (1-\theta) G^{N-1}(X^{N-1}) \rt)dX^{N-1} \int_{|x-y|\geq R} |x-y|^{-(2+a)}  \Phi(\psi_\e(x),\psi_\e(y))dxdy\\
		&\quad -\int_{\R^{2(N-1)}} \int_{|x-y|\leq R}\int_0^1 \frac{ (1-\theta)\theta G^{N-1}F^{N-1}\lt( f(x|)-f(y|) \rt) g(z_t|.)\nabla \ln\frac{g(|.)}{ f(|.)}(z_t)\cdot (x-y) }{|x-y|^{2+a}\lt(\theta f(z_t|.)F^{N-1} +(1-\theta)g(z_t|.)G^{N-1}\rt)}dtdxdydX^{N-1}\\
		&\quad 	-\int_{\R^{2(N-1)}} \int_{|x-y|\leq R}\int_0^1 \frac{ (1-\theta)\theta G^{N-1}F^{N-1}\lt( g(x|)-g(y|) \rt) f(z_t|.)\nabla \ln\frac{f(|.)}{ g(|.)}(z_t)\cdot (x-y) }{|x-y|^{2+a}\lt(\theta f(z_t|.)F^{N-1} +(1-\theta)g(z_t|.)G^{N-1}\rt)}dtdxdydX^{N-1}.
		\end{split}
		\end{equation} 
		
		Since, by \eqref{eq:loin} it holds
		\begin{align*}
		\int_{|x-y|\geq R} |x-y|^{-(2+a)}  \Phi(\psi_\e(x),\psi_\e(y))dxdy & \leq \int_{|x-y|\geq R} |x-y|^{-(1+a)}  (\psi_\e(x)+\psi_\e(y))dxdy\\
		&\leq 2 \int_{\R^2} \psi_\e(x)\lt(\int_{y, |x-y|>R} |x-y|^{-(1+a)}dy\rt)dx<\eta
		\end{align*}	
		choosing $R$ sufficiently large yields
			\begin{equation}
			\begin{split}
			\label{eq:aff}
			\mathcal{K}^N&\leq \eta-\int_{\R^{2(N-1)}} \int_{|x-y|\leq R}\int_0^1 \frac{ (1-\theta)\theta G^{N-1}F^{N-1}\lt( f(x|)-f(y|) \rt) g(z_t|.)\nabla \ln\frac{g(|.)}{ f(|.)}(z_t)\cdot (x-y) }{|x-y|^{2+a}\lt(\theta f(z_t|.)F^{N-1} +(1-\theta)g(z_t|.)G^{N-1}\rt)}dtdxdydX^{N-1}\\
			&\quad 	-\int_{\R^{2(N-1)}} \int_{|x-y|\leq R}\int_0^1 \frac{ (1-\theta)\theta G^{N-1}F^{N-1}\lt( g(x|)-g(y|) \rt) f(z_t|.)\nabla \ln\frac{f(|.)}{ g(|.)}(z_t)\cdot (x-y) }{|x-y|^{2+a}\lt(\theta f(z_t|.)F^{N-1} +(1-\theta)g(z_t|.)G^{N-1}\rt)}dtdxdydX^{N-1}.
			\end{split}
			\end{equation} 		
	 Let now be $s\in (0,r)$ and define 
		$$F'=\mb_{\mathcal{B}_s}F, \quad \mbox{and} \quad F''=F-F'. $$
		Then 
		\begin{equation*}
			\begin{split} 
				\mathcal{K}^N&\leq \eta -\int_{\R^{2(N-1)}}\int_{|x-y|\leq R}\int_0^1 \frac{ (1-\theta)\theta G^{N-1}F'^{N-1}\lt( f(x|)-f(y|) \rt) g(z_t|.)\nabla \ln\frac{g(|.)}{ f(|.)}(z_t)\cdot (x-y) }{|x-y|^{2+a}\lt(\theta f(z_t|.)F'^{N-1} +(1-\theta)g(z_t|.)G^{N-1}\rt)}dtdxdydX^{N-1}\\
				&\quad - \int_{\R^{2(N-1)}}\theta F''^{N-1} \int_{|x-y|\leq R}\lt| f(x|)-f(y|) \rt|\int_0^1 \frac{ g(z_t) \lt|\nabla \ln\frac{g(|.)}{ f(|.)}(z_t)\rt| }{f(z_t)|x-y|^{1+a}}dtdxdydX^{N-1}\\
				&\quad 	-\int_{\R^{2(N-1)}}\int_{|x-y|\leq R}\int_0^1 \frac{ (1-\theta)\theta G^{N-1}F'^{N-1}\lt( g(x|)-g(y|) \rt) f(z_t|.)\nabla \ln\frac{f(|.)}{ g(|.)}(z_t)\cdot (x-y) }{|x-y|^{2+a}\lt(\theta f(z_t|.)F'^{N-1} +(1-\theta)g(z_t|.)G^{N-1}\rt)}dtdxdydX^{N-1}\\
				&\quad - \int_{\R^{2(N-1)}}\theta F''^{N-1} \int_{|x-y|\leq R}\lt| g(x|)-g(y|) \rt|\int_0^1 \frac{ \lt|\nabla \ln\frac{g(|.)}{ f(|.)}(z_t)\rt| }{|x-y|^{1+a}}dtdxdydX^{N-1}\\
				&\quad := \mathcal{K}^N_1+\mathcal{K}^N_2+\mathcal{K}^N_3+\mathcal{K}^N_4.
			\end{split}	
		\end{equation*}

	Set now $u=\frac{r+s}{2}$ and $\delta=\frac{r-s}{2}$, and denote
	\[
	\tilde{\B}_u^{N-1}=\lt\{(x_2,\cdots,x_N)\in \R^{2(N-1)}\, | \, \frac{1}{N-1}\sum_{i=2}^N\delta_{x_i}\in \mathcal{B}_u \rt\},
	\]
	then
	\begin{align*}
	\mathcal{K}^N_1&\leq \int_{\R^{2(N-1)}}\int_{|x-y|\leq R}\lt(\mb_{\B^{N-1}}+\mb_{\B^{N-1,c}}\rt)\int_0^1 \frac{ (1-\theta)\theta G^{N-1}F'^{N-1}\lt| f(x|)-f(y|) \rt| g(z_t|.)\lt|\nabla \ln\frac{g(|.)}{ f(|.)}(z_t)\rt| }{|x-y|^{1+a}\lt(\theta f(z_t|.)F'^{N-1} +(1-\theta)g(z_t|.)G^{N-1}\rt)}dtdxdydX^{N-1}\\
	& \leq \int_{\R^{2(N-1)}}(1-\theta)\mb_{\B^{N-1}}G^{N-1}\int_{|x-y|\leq R}\lt| f(x|)-f(y|) \rt|\int_0^1 \frac{  g(z_t|.)\lt|\nabla \ln\frac{g(|.)}{ f(|.)}(z_t)\rt| }{|x-y|^{1+a}f(z_t|.)}dtdxdydX^{N-1}\\
	&+\int_{\R^{2(N-1)}}\theta \mb_{\B^{N-1,c}}F'^{N-1}\int_{|x-y|\leq R}\lt| f(x|)-f(y|) \rt|\int_0^1 \frac{  \lt|\nabla \ln\frac{g(|.)}{ f(|.)}(z_t)\rt| }{|x-y|^{1+a}}dtdxdydX^{N-1}
	\end{align*}
	Using Lemma \ref{lem:reg} we find easily that
	\[
	\lt| f(x|)-f(y|) \rt|\int_0^1 \frac{  \lt|\nabla \ln\frac{g(|.)}{ f(|.)}(z_t)\rt| }{|x-y|^{1+a}}dt\leq 2\e^{-1}\lt| f(x|)-f(y|) \rt||x-y|^{-(1+a)}
	\]
	Therefore
	\begin{align*}
	\int_{|x-y|\leq R}\lt| f(x|)-f(y|) \rt|\int_0^1 \frac{  \lt|\nabla \ln\frac{g(|.)}{ f(|.)}(z_t)\rt| }{|x-y|^{1+a}}dtdxdy & \leq \int_{|x-y|\leq R}\lt|\ln f(x|)-\ln f(y|)\rt|\frac{\lt|f(x|)-f(y|)\rt|}{\lt|\ln f(x|)-\ln f(y|)\rt|}|x-y|^{-(1+a)}dxdy\\
	&\leq \int_{|x-y|\leq R}\lt( f(x|)+f(y|)\rt)\| \nabla \ln f(|) \|_{L^\infty}|x-y|^{-a}dxdy \\
	&\leq 2\e^{-1}\int_{\R^2}f(x|) \lt(\int_{y\in \R^2, |x-y|>R}|x-y|^{-a}dy \rt) dx\leq C_{\e,R,a}
	\end{align*}
	On the other hand
	\begin{align*}
	\lt| f(x|)-f(y|) \rt|\int_0^1 \frac{  g(z_t|.)\lt|\nabla \ln\frac{g(|.)}{ f(|.)}(z_t)\rt| }{|x-y|^{1+a}f(z_t|.)}dt&\leq 2 \e^{-1}\lt| f(x|)-f(y|) \rt|\int_0^1 \frac{  g(z_t|.) }{|x-y|^{1+a}f(z_t|.)}dt\\
	&\quad \leq 2 \e^{-1} \lt|\int_0^1 \nabla f(z_t|)\cdot (x-y)dt\rt|\int_0^1 \frac{  g(z_t|.) }{|x-y|^{1+a}f(z_t|.)}dt\\
	&\quad \leq 2 \e^{-1} |x-y|^{-a} \lt( \int_0^1 \sqrt{ \frac{ \lt|\nabla f(z_t|)\rt| g(z_t|.) }{f(z_t|.)}  } dt \rt)^2\\
	&\quad \leq 2\e^{-1}|x-y|^{-a} \lt( \int_0^1 \sqrt{ \lt|\nabla \ln f(z_t|)\rt| g(z_t|.) }   dt \rt)^2\\
	&\quad	\leq 2\e^{-2}|x-y|^{-a}\lt( \int_0^1 \sqrt{  g(z_t|.) }   dt \rt)^2\leq 2\e^{-2}|x-y|^{-a} \int_0^1 g(z_t|.)dt, 
	\end{align*}
	where we used point $(i)$ of Lemma \ref{lem:reg} to pass to the last line. But
	\begin{align*}
	\int_0^1 g(z_t|.)dt=\int_0^1 g\lt(y+t|x-y|\frac{(x-y)}{|x-y|}|.\rt)dt
	\end{align*}
	for $|x-y|\leq R$ by we have by point $(ii)$ of Lemma \ref{lem:reg}
	\[
	\int_0^1 g(z_t|.)dt\leq C_{\e,R} g(y|.)
	\]
	So that 
	\begin{align*}
		\int_{|x-y|\leq R}\lt| f(x|)-f(y|) \rt|\int_0^1 \frac{  g(z_t|.)\lt|\nabla \ln\frac{g(|.)}{ f(|.)}(z_t)\rt| }{|x-y|^{1+a}f(z_t|.)}dt\leq 2\e^{-2}C_{\e,R}\int_{|x-y|\leq R}g(y)||x-y|^{-a}dxdy\leq C_{\e,a,R}
	\end{align*}	
	
	Finally 
	\begin{align*}
		\mathcal{K}^N_1\leq  C_{\e,a,R}\lt((1-\theta)\int_{\R^{2(N-1)}}\mb_{\B^{N-1}}G^{N-1}dX^{N-1}+\theta \int_{\R^{2(N-1)}}\mb_{\B^{N-1,c}}F'^{N-1}dX^{N-1}\rt)
	\end{align*}
	
	The other terms are treated similarly and we conclude this step with
	\begin{align*}
		\mathcal{K}^N& \leq \eta + C_{\e,a,R}\lt((1-\theta)\int_{\R^{2(N-1)}}\mb_{\B^{N-1}}G^{N-1}dX^{N-1}+\theta \int_{\R^{2(N-1)}}\mb_{\B^{N-1,c}}F'^{N-1}dX^{N-1}\rt)\\
		& + C_{\e,a,R}\theta \int_{\R^{2(N-1)}} F''^{N-1}dX^{N-1}
	\end{align*}
		
	$\diamond$ \textit{Step five}\newline
	The end of the proof is then exactly taken from \cite[Lemma 5.10]{HM}. Nevertheless we reproduce it here for the sake of completeness. First we treat the third trerm in the above r.h.s. by observing that
	$F^{''}=\mb_{\mathcal{B}_r\setminus \mathcal{B}_s}F$. Therefore
	\[
	\begin{aligned}
	\int_{\R^{2(N-1)}} F''^{N-1}(X^{N-1}) dX^{N-1}=\int_{\PP(\R^2)} \mb_{\mathcal{B}_r\setminus \mathcal{B}_s}(\rho_\e)F(d\rho).
	\end{aligned}
	\]
	Due to Lebesgue's dominated convergence Theorem, the r.h.s. in the above identity goes to $0$. Therefore one can chose some $s<r$ such that 
	\[
	C_\e c_a \theta\int_{\R^{2(N-1)}} F''^{N-1}(X^{N-1}) dX^{N-1}<\eta,
	\]
	uniformly in $N$. Then for $X^{N-1}\notin \tilde{\B}_u^{N-1}$ and $\rho\in \mathcal{B}_s$ we find that
	\[
	\begin{aligned}
	W_1\lt(\frac{1}{N-1}\sum_{i=2}^N\delta_{x_i},\rho*\psi_\e\rt)&\geq 	W_1\lt(\frac{1}{N-1}\sum_{i=2}^N\delta_{x_i},f_1\rt)-	W_1\lt(f_1,\rho\rt)-	W_1\lt(\rho,\rho*\psi_\e\rt)\\
	&\quad \geq u-s-c\e\geq \frac{\delta}{2},
	\end{aligned}
	\]
	for any $\e>0$ small enough. Therefore using a Chebychev-like argument it holds
	\[
	\begin{aligned}
	\int_{\tilde{\B}_u^{N-1,c}}F'^{N-1}(X^{N-1})dX^{N-1}&=\int_{\PP(\R^2)}\lt(\int_{\R^{2(N-1)}}\mb_{\tilde{\B}_u^{N-1,c}} \rho_\e^{\otimes(N-1)}\rt) F'(d\rho)\\
	&\quad \leq \frac{2}{\delta}\int_{\PP(\R^2)}\lt(\int_{\R^{2(N-1)}} W_1\lt(\frac{1}{N-1}\sum_{i=2}^N\delta_{x_i},\rho_\e\rt)\rho_\e^{\otimes(N-1)}(dX^{N-1})\rt) F'(d\rho)
	\end{aligned}
	\]
	We claim that there is a constant $C$ depending only on $\kappa$ (see \cite[Theroem 1]{FG} in case $d=2,p=1,q=\kappa<2$) such that it holds
	\bq
	\label{eq:concHM}
	\int_{\R^{2(N-1)}} W_1\lt(\frac{1}{N-1}\sum_{i=2}^N\delta_{x_i},\rho_\e\rt)\rho_\e^{\otimes(N-1)}(dX^{N-1})\leq C \lt(M_{\kappa}(\rho_\e)\rt)^{\frac{1}{\kappa}} (N-1)^{-\lt(1-\frac{1}{\kappa}\rt)}.
	\eq
	Note that \cite[Remark 2.12]{HM} provides the same result with the exponent $1-\frac{1}{\kappa}$ replaced with $\gamma \in \lt(0,\frac{1}{3+\frac{2}{\kappa}}\rt)$, but the rate of convergence does not play any role in the proof. Summing up \eqref{eq:concHM} w.r.t. $F'$, yields
	\[
	\begin{aligned}
	\int_{\tilde{\B}_u^{N-1,c}}F'^{N-1}(X^{N-1})dX^{N-1}&\leq \frac{C}{\delta(N-1)^{\lt(1-\frac{1}{\kappa}\rt)}}\int_{\PP(\R^2)} \lt(M_{\kappa}(\rho_\e)\rt)^{\frac{1}{\kappa}} F'(d\rho)\\
	&\quad \leq \frac{C}{\delta(N-1)^{\lt(1-\frac{1}{\kappa}\rt)}}\lt(\int_{\PP(\R^2)} M_{\kappa}(\rho)\pi(d\rho)+M_{\kappa}(\psi_\e)\rt)^{\frac{1}{\kappa}},
	\end{aligned}
	\]
	since
	\[
	\begin{aligned}
	M_{\kappa}(\rho_\e)&=\int_{\R^2}\int_{\R^2} \lal x \ral^{\kappa}\rho(x-y)\psi_\e(y)dxdy = \int_{\R^2}\int_{\R^2} \lal x+y \ral^{\kappa}\rho(x)\psi_\e(y)dxdy\\
	&\quad \leq 2^\kappa \lt(\int_{\R^2}\lal x\ral^\kappa\rho(x)dx+\int_{\R^2}\lal y\ral^\kappa\psi_\e(y)dy\rt)
	\end{aligned}
	\]
	
	Treating in the exact same fashion the integral w.r.t. $G^{N-1}$ concludes this step with
	\[
	\forall \eta>0,\, \exists N_\eta, \, \mbox{s.t.} \, \forall N \geq N_{\eta}, \, \mathcal{K}^N\leq \eta.
	\]
	$\diamond$ Final step\newline
	Gathering all the estimates obtained in the previous steps yields for any $\e>0$
	\[
	\lim_{N\rightarrow +\infty }\lt|\II^N_a(\pi^{N,\e})-\theta\II^N_a(F^{N,\e})-(1-\theta)\II^N_a(G^{N,\e})\rt|=0.
	\]
	Hence we deduce
	\[
	\begin{aligned}
	\tilde{\II}_a(\pi^\e)&=\sup_{N\geq 1}\II^N_a(\pi^{N,\e})=\lim_{N\rightarrow+\infty}\II^N_a(\pi^{N,\e})\\
	&\quad = \theta \lim_{N\rightarrow+\infty}\II^N_a(F^{N,\e})+(1-\theta)\lim_{N\rightarrow+\infty}\II^N_a(G^{N,\e})\\
	&\quad =\theta \sup_{N\geq 1}\II^N_a(F^{N,\e})+(1-\theta)\sup_{N\geq 1}\II^N_a(G^{N,\e})\\
	&\quad = \theta \tilde{\II}_a(F^\e)+(1-\theta)\tilde{\II}_a(G^\e).
	\end{aligned}
	\]
	
	But using the convexity of the functional $\Phi$ and Jensen's inequality yields 
	\[
	\begin{aligned}
	\tilde{\II}_a(\pi^\e)=\sup_{N\geq 1}\II_a^N(\pi^{N,\e})\leq \sup_{N\geq 1}\II_a^N(\pi^{N})=\tilde{\II}_a(\pi). 	
	\end{aligned}
	\] 
	Morever it is clear from the fact that the functionals $(\II^N_a)_{N\geq 1}$ are l.s.c.  w.r.t. the weak convergence in $\PP(\R^{2N})$, that $\tilde{\II}_a$ is l.s.c.  w.r.t. the weak convergence in $\PP(\PP(\R^2))$. But since $\pi^\e{\overset{*}{\rightharpoonup}}\pi$ in $\PP(\PP(\R^2))$ we get that
	\[
	\lim_{\e\rightarrow 0} \tilde{\II}_a(\pi^\e)=\tilde{\II}_a(\pi).
	\]
	Therefore 
	\[
	\tilde{\II}_a(\pi)=\theta\tilde{\II}_a(F)+(1-\theta)\tilde{\II}_a(G),
	\]
	which concludes the proof.

\section*{Acknowledgements}
The author was supported by the Fondation des Sciences Math\'ematiques de Paris and Paris Sciences \& Lettres Universit\'e, and warmly thanks Maxime Hauray and St\'ephane Mischler for many advices, comments and discussions which have made this work possible. Also many thanks to Laurent Lafl\`eche for his collaboration on a forthcoming joint work, and which has enabled to simplify among others the proof of Lemma \ref{lem:GNS}. Also, I have corrected a gap in the proof of Lemma \ref{lem:aff}, by using the convex combination inside the argument of the ratio $g/f$ in the proof of \eqref{eq:id}, (instead of the convex combination of the ratio $g/f$ in the previous version).
%
%
%
%

%
%
%
%

\end{document}